\documentclass[12pt,oneside,english]{amsart}
\usepackage[T1]{fontenc}
\usepackage[latin9]{inputenc}
\usepackage{geometry}
\usepackage{color}
\usepackage{textcomp}
\usepackage{mathrsfs}
\usepackage{amstext}
\usepackage{amsthm}
\usepackage{amssymb}
\usepackage{cancel}
\usepackage{graphicx}

\makeatletter
\numberwithin{equation}{section}
\numberwithin{figure}{section}
\theoremstyle{plain}
\newtheorem{thm}{\protect\theoremname}
\theoremstyle{remark}
\newtheorem*{rem*}{\protect\remarkname}
\theoremstyle{remark}
\newtheorem{rem}[thm]{\protect\remarkname}
\theoremstyle{plain}
\newtheorem{prop}[thm]{\protect\propositionname}
\theoremstyle{plain}
\newtheorem{fact}[thm]{\protect\factname}

\makeatother

\usepackage{babel}
\providecommand{\factname}{Fact}
\providecommand{\propositionname}{Proposition}
\providecommand{\remarkname}{Remark}
\providecommand{\theoremname}{Theorem}

\begin{document}
\title{Rethinking the work of Langlands on Eisenstein series}
\author{Devadatta G. Hegde}
\begin{abstract}
Chapter $7$ of Langlands' monograph \emph{On the functional equations
satisfied by Eisenstein series }employs a sophisticated residue scheme
to construct a portion of the discrete automorphic spectrum. We show,
by examples, applications, and heuristics, that this construction
is a straightforward regularization of cuspidal Eisenstein series
at distinguished points, and that the regularization must track \emph{both
the zeros and the poles} of these Eisenstein series. 

Unlike the one-variable case, the zero set and pole set of a several-complex-variable
meromorphic function can intersect at a point. The distinguished points
supporting the discrete spectrum are typically of this kind. The zeros
of cuspidal Eisenstein series --- largely invisible in the rank-one
case --- begin to play a starring role in higher rank situations,
on equal footing with the poles. We redo Langlands' famous $G_{2}$
calculation and show that once zeros are tracked, the calculation
reduces to elementary algebra. 

Drawing on rank-two examples, we introduce a program to re-think Langlands'
construction from first principles, giving zeros and poles of cuspidal
Eisenstein series equal standing from the very beginning. The program
has the advantage of making the underlying phenomenon transparent,
though carrying it out in full generality will require substantial
further work.
\end{abstract}

\maketitle
\tableofcontents{}

\section*{Organization of the notes}

These notes are organized so that each section can be read on its
own, with the $SL(2)$, $SL(3)$, and $G_{2}$ examples (sections
\ref{sec:SL2}, \ref{sec:SL3}, \ref{sec:G2}) forming the technical
core. 

\bigskip{}

Section \ref{sec:new-results} describes two theorems on the poles
of certain Eisenstein series. The first, due to the author \cite{Dattu-thesis},
concerns the poles of degenerate Eisenstein series. The second concerns
the poles of the most general $GL(n)$ Eisenstein series relevant
for spectral decomposition; it can be derived by the methods of the
author's thesis, and is a reformulation, in more natural terms, of
an obscure theorem of Moeglin--Waldspurger \cite{Moeglin-Waldspurger-89}.
Together these results suggest a new theme: that the poles of Eisenstein
series admit simple descriptions in terms of Arthur parameters. Looking
for the origin and generalization of this pattern is what opened the
door to a re-examination of Langlands' construction of the residual
spectrum.

\bigskip{}

Section \ref{sec:SL2} rethinks the basic case of $SL(2,\mathbb{Z})\backslash SL(2,\mathbb{R})/SO(2,\mathbb{R})$.
The goal is to introduce the contour-deformation argument in its simplest
setting and to clarify the relationship between Laurent coefficients
of Eisenstein series at poles and spectral expansions. Carrying out
the deformation under the a priori assumption that the poles need
not be simple, one finds that simplicity is forced: without it, the
deformation does not yield a spectral expansion. This gives a fresh
perspective on the standard result that maximal-parabolic cuspidal
Eisenstein series of reductive groups have simple poles in the right
half-plane, and identifies a relationship --- between Laurent coefficients
and spectral expansions --- whose lessons apply more broadly.

\bigskip{}

Section \ref{sec:SL3} reworks the $SL(3)$ calculation from Langlands'
Boulder paper \cite{Langlands_Boulder}. The goal is to carry out
higher-rank contour deformation in the simplest case, illustrate the
role of zeros of Eisenstein series, and clarify a certain \textquotedbl turbid
lemma\textquotedbl{} of Langlands. We treat this example completely,
including the one-dimensional spectrum, so that the next section can
focus on the discrete spectrum of $G_{2}$ alone. The justification
of the contour deformation requires a certain cancellation; a critical
look at this cancellation, in \S \ref{subsec:Critical-remarks-SL3},
suggests that the locations of the poles of cuspidal Eisenstein series
are sharply restricted.

\bigskip{}

Section \ref{sec:G2} reworks the $G_{2}$ calculation from Appendix
III of \cite{Langlands-Spectral}. Langlands offered the example to
help the reader navigate Chapter 7 of his monograph, and observed
in his commentary that the complications arising in it could not be
eliminated \textquotedbl unless a procedure radically different from
that of the text be found.\textquotedbl{} We carry out the calculation
by tracking both the zeros and the poles of the corresponding Borel
Eisenstein series. With this addition, the calculation proceeds by
elementary algebra. The two-dimensional residual spectrum identified
by Langlands is recovered in the form one expects from the $SL(3)$
analysis of the previous section: as regularizations of the Borel
Eisenstein series at two distinguished points.

\bigskip{}

Section \ref{sec:General} outlines a program for rethinking the \textquotedbl jungle
of Paragraph \S7\textquotedbl{} with equal attention paid to the zeros
and poles of cuspidal Eisenstein series from the outset. 

\newpage{}

\section{\label{sec:new-results}New results on poles of Eisenstein series\protect 
}
\subsection{\label{subsec:deg-Eis-poles}Poles of degenerate Eisenstein series }

\subsubsection*{Notation}

Let
\begin{itemize}
\item $G$ be a connected split semisimple algebraic group over a number
field $F$. 
\item $B=U\rtimes T$ be a fixed Levi decomposition of a fixed Borel subgroup
of $G$. 
\item $P=N\rtimes M$, with $T\subset M$, be the standard Levi decomposition
of a \emph{maximal} parabolic subgroup $P\supset B$ of $G$. 
\item $[G]=G(F)\backslash G(\mathbb{A})$ be the automorphic space.
\item $\delta_{P}:P(\mathbb{A})\to(0,\infty)$ be the modulus character. 
\item $\varpi_{P}$ be the fundamental weight associated with $P$ and 
\[
|\varpi_{P}|:P(\mathbb{A})\to(0,\infty)
\]
be the corresponding character. 
\item $\mathbf{K} = K_\infty \prod_{v<\infty} K_v \subset G(\mathbb{A})$:
a maximal compact subgroup. Here, each $K_v$ at a finite place $v$ is hyperspecial and the Iwasawa decomposition
\[
G(\mathbb{A})=P(\mathbb{A})\mathbf{K}
\]
holds. 
\end{itemize}

\subsubsection*{Degenerate Eisenstein series}

Extend 
\[
\delta_{P},|\varpi_{P}|:G(\mathbb{A})\to(0,\infty)
\]
as $\mathbf{K}$-spherical functions. The degenerate Eisenstein series
\[
E_{P}(s;g)=\sum_{\gamma\in P(F)\backslash G(F)}\left(\delta_{P}^{\frac{1}{2}}\cdot|\varpi_{P}|^{s}\right)(\gamma g)
\]
is the simplest example of an Eisenstein series. This series converges
absolutely for $\Re(s)$ sufficiently large and uniformly on compact
subsets of $[G]$. The map 
\[
s\mapsto E_{P}(s):=E_{P}(s;\bullet)\in C^{\infty}\left([G]\right)
\]
admits meromorphic continuation to $\mathbb{C}$. 

\subsubsection*{Poles}

The poles of $E_{P}(s)$ in the region $\Re(s)\ge0$ are described
via the \emph{principal homomorphism }
\[
\kappa:SL(2,\mathbb{C})\to\widehat{M}
\]
and the adjoint action of $\widehat{M}$ on $\widehat{\mathfrak{n}}=\text{Lie}(\widehat{N})$,
where $\widehat{\bullet}$ denotes the Langlands dual. 

Decompose (Langlands-Shahidi) 
\[
\widehat{\mathfrak{n}}=r_{1}\oplus\cdots\oplus r_{m}
\]
into irreducible representations $r_{j}$ of $\widehat{M}$ indexed
in the standard order. Restricting via $\kappa$, 
\[
r_{j}=\bigoplus_{\ell\ge1}V_{\ell}^{\oplus m_{j}(\ell)},\quad m_{j}(\ell)\in\mathbb{Z}_{\ge0}
\]
where $V_{\ell}$ is the $\ell$-dimensional irreducible representation
of $SL(2,\mathbb{C})$. The following theorem is proved in the author's
doctoral thesis \cite{Dattu-thesis}. 
\begin{thm}
\label{thm:deg-Eis}Define
\[
\mathsf{Q}(s)=\prod_{j=1}^{m}\prod_{\ell\ge1}\left(\frac{js-\frac{\ell+1}{2}}{js+\frac{\ell-1}{2}}\right)^{m_{j}(\ell)}
\]
Then $\mathsf{Q}(s)\cdot E_{P}(s)$ is \textbf{holomorphic and nonzero}
for $\Re(s)\ge0$ (hence the zeros of $\mathsf{Q}$ give the poles
of $E_{P}(s)$, including multiplicities). 
\end{thm}

\subsection{\label{subsec:GLn-poles}Poles of $GL(n)$ Eisenstein series}

Let $(n_{1},\dots,n_{k})$ be a partition of $n=n_{1}+\cdots+n_{k}$,
with $n_{i}=a_{i}d_{i}$ for $a_{i},d_{i}\in\mathbb{Z}_{\ge1}$. 

To show that the above result is not isolated, we describe the poles
of all $GL(n)$ Eisenstein series that play a role in the spectral
decomposition of $GL(n)$ automorphic forms. 
\begin{itemize}
\item $G=GL(n)$ and $P=N\rtimes M$ be the standard Levi decomposition
of the parabolic corresponding to $(n_{1},\dots,n_{k})$. Then 
\[
\widehat{M}=GL(n_{1},\mathbb{C})\times\cdots\times GL(n_{k},\mathbb{C})
\]
 
\item For $a,d\in\mathbb{Z}_{\ge1}$ and $\tau$ a cuspidal representation
of $GL(d,\mathbb{A})$, let $\mathcal{E}(\tau,a)$ be the Speh representation
for $GL(ad,\mathbb{A})$ with Arthur parameter 
\[
A_{\mathcal{E}(\tau,a)}=A_{\tau}\boxtimes V_{a}:\mathscr{L}_{F}\times SL(2,\mathbb{C})\to GL(ad,\mathbb{C})
\]
where $\mathscr{L}_{F}$ is the (conjectural) Langlands group and
$V_{a}$ is the $a$-dimensional irreducible representation of $SL(2,\mathbb{C})$.
Here $A_{\tau}:\mathscr{L}_{F}\to GL(d,\mathbb{C})$ is the (conjectural)
Langlands parameter attached to $\tau$. 
\item Let $\pi\simeq\boxtimes_{i=1}^{k}\mathcal{E}(\tau_{i},a_{i})$ where
$\tau_{i}$ is a cuspidal representation of $GL(d_{i},\mathbb{A})$.
Its Arthur parameter is 
\[
A_{\pi}=\bigoplus_{i=1}^{k}A_{\mathcal{E}(\tau_{i},a_{i})}:\mathscr{L}_{F}\times SL(2,\mathbb{C})\to\widehat{M}
\]
\end{itemize}
The adjoint action gives 
\[
\widehat{\mathfrak{n}}=\bigoplus_{(i,j):1\le i<j\le k}\widehat{\mathfrak{n}_{ij}},
\]
where
\[
\widehat{\mathfrak{n}_{ij}}\simeq\left(A_{\tau_{i}}\boxtimes V_{a_{i}}\right)\otimes\left(A_{\tau_{j}}\boxtimes V_{a_{j}}\right)^{\vee}
\]
\[
\simeq\left(A_{\tau_{i}}\otimes A_{\tau_{j}}^{\vee}\right)\boxtimes\left(V_{a_{i}}\otimes V_{a_{j}}\right).
\]
By Clebsch-Gordan,
\[
V_{a_{i}}\otimes V_{a_{j}}\simeq\bigoplus_{\ell\ge1}V_{\ell}^{\oplus m_{ij}(\ell)}
\]
where $m_{ij}(\ell)\in\mathbb{Z}_{\ge0}$. 

Let $\alpha_{ij}^{\vee}$ denote the coroot associated to the pair
$(i,j)$. For $\Lambda=(s_{1},\dots,s_{k})\in\check{\mathfrak{a}}_{P,\mathbb{C}}$,
set
\[
\langle\Lambda,\alpha_{ij}^{\vee}\rangle:=s_{i}-s_{j}
\]
and define the closed positive tube
\[
T_{P}:=\left\{ \Lambda\in\check{\mathfrak{a}}_{P,\mathbb{C}}:\text{Re}\langle\Lambda,\alpha_{ij}^{\vee}\rangle\ge0\text{ for all }1\le i<j\le k\right\} 
\]
The following theorem can be proved using the argument in \cite{Dattu-thesis}.
It was recently pointed out to me that this theorem is contained in
Lemme III.2 of \cite{Moeglin-Waldspurger-89} in a different formulation,
and that few people seem to have been aware of this result \cite{Hanzer_Muic}. 
\begin{thm}
\label{thm:GL(n)}Let
\[
D_{P}(\pi,\Lambda):=\prod_{\substack{(i,j):1\le i<j\le k\\
\tau_{i}\simeq\tau_{j}
}
}\prod_{\ell\ge1}\left(\frac{\langle\Lambda,\alpha_{ij}^{\vee}\rangle-\frac{\ell+1}{2}}{\langle\Lambda,\alpha_{ij}^{\vee}\rangle+\frac{\ell-1}{2}}\right)^{m_{ij}(\ell)}
\]
The intertwining operator 
\[
D_{P}(\pi,\Lambda)\cdot E_{P}(\pi,\Lambda):\mathrm{Ind}_{P(\mathbb{A})}^{G(\mathbb{A})}(\pi,\Lambda)\to\mathcal{A}([G])
\]
is \textbf{holomorphic and nonzero} for $\Lambda\in T_{P}$.
\end{thm}

\subsubsection*{Simultaneous occurrence of zeros and poles of cuspidal Eisenstein
series}

Consider the special case of the above theorem where $\pi=\tau_{1}\boxtimes\cdots\boxtimes\tau_{k}$
is a cuspidal representation of $M(\mathbb{A})$, so all $a_{i}=1$.
Then each Clebsch-Gordan decomposition is $V_{a_{i}}\otimes V_{a_{j}}=V_{1}$
has only the $\ell=1$ term, and theorem \ref{thm:GL(n)} specializes
to the following statement: $E_{P}(\pi,\Lambda)$ has the following
property
\[
\text{pole along }\langle\Lambda,\alpha_{ij}^{\vee}\rangle=1\ \iff\ \text{zero along }\langle\Lambda,\alpha_{ij}^{\vee}\rangle=0\ \iff\ \tau_{i}\simeq\tau_{j}.
\]
The first implication is a result of Keys and Shahidi \cite{Keys-Shahidi};
it plays a fundamental role in the proof of Theorem \ref{thm:GL(n)}
following \cite{Dattu-thesis}. 

In the author's view, this phenomenon --- the simultaneous occurrence
of a zero and a pole of a cuspidal Eisenstein series --- is central
in the theory of Eisenstein series and in spectral decomposition via
contour deformation.

\subsection{Remarks}
\begin{itemize}
\item The formulas in Theorems \ref{thm:deg-Eis} and \ref{thm:GL(n)} appear
to be the first instances of a general pattern in the theory of Eisenstein
series. The key input to prove these theorems is an \emph{explicit}
residue formula that expresses non-cuspidal Eisenstein series as residues
of cuspidal Eisenstein series. If one expects these patterns, occurring
in quite different situations, to hold generally, then one must expect
an explicit residue formula in general. This is the origin of the
present re-examination.
\item The most natural extension of these theorems requires that the contribution
to the discrete spectrum from cuspidal Eisenstein series occur as
a simple regularization. See \ref{subsubsec:disc-examples} for further
examples.
\item The local counterpart of poles of Eisenstein series is the reducibility
of normalized parabolic induction for reductive groups over local
fields. Here, the Arthur parameter should be replaced by the Langlands
parameter.
\end{itemize}
\newpage{}

\section{\label{sec:SL2}The simplest case: $SL(2,\mathbb{Z})\backslash SL(2,\mathbb{R})/SO(2,\mathbb{R})$}

The goal of this section is to elaborate the introductory remarks
that follow. 

Let 
\[
\Gamma=SL(2,\mathbb{Z}),\quad G=SL(2,\mathbb{R}),\quad K=SO(2,\mathbb{R}),
\]
\[
\mathfrak{H}=\left\{ z\in\mathbb{C}:\Im(z)>0\right\} \simeq G/K,
\]
and 
\[
\Delta=-y^{2}(\partial_{x}^{2}+\partial_{y}^{2}).
\]

\subsubsection*{The simplest $SL(2)$-Eisenstein series}

For $z\in\mathfrak{H}$ and $\text{Re}(s)>1$, let
\[
E(s;z)=\frac{1}{2}\sum_{(c,d)\in\mathbb{Z}^{2}:\text{gcd}(c,d)=1}\frac{y^{\frac{1+s}{2}}}{|cz+d|^{1+s}}
\]
The map 
\[
s\mapsto E(s):=E(s;\bullet)\in C^{\infty}(\Gamma\backslash\mathfrak{H})
\]
has a meromorphic continuation to $\mathbb{C}$, and 
\[
\Delta E(s)=\lambda(s)E(s),\quad\lambda(s):=\frac{1-s^{2}}{4}.
\]

\subsubsection*{The standard approach}

Maass--Selberg relation \cite{Garrett1} is used to prove the standard
fact that the poles of $E(s)$ in the region $\Re(s)>0$ are real
and simple. The residues at these poles are $\Delta$-eigenfunctions
in $L^{2}(\Gamma\backslash\mathfrak{H})$ and are picked up by Langlands'
contour deformation. The picture extends to maximal parabolic Eisenstein
series induced from cuspidal data, and is essentially the only case
in which Langlands' \textquotedbl complicated residue scheme\textquotedbl{}
is understood (see \cite{Lapid-Perspectives-pub}, 5.1).

\subsubsection*{Our re-examination}

We propose to read this differently. Rather than taking simplicity
as an external input to the method, we ask what happens without it:
the method of contour deformation fails to produce a spectral expansion.
The rank-one slogan is:
\[
\text{Maass-Selberg}\implies\text{Simplicity of poles of }E(s)\text{ for }\Re(s)>0
\]
\[
\implies\text{contour deformation works.}
\]
The same heuristic controls \S \S \ref{sec:SL3}, \ref{sec:G2}, with one essential change: in
several complex variables, the zero set and pole set of a meromorphic
function can intersect, and the distinguished points contributing
to the spectrum are usually of this kind. The contour deformation
there picks up a regularization of the Eisenstein series --- sensitive
to both zeros and poles.

\subsubsection*{An important remark}

A further constraint: positive-definiteness of $\Delta$ confines
the poles of $E(s)$ in the right half-plane to $(0,1]$, but inductive
arguments at $\text{rank}\ge2$ require that no poles lie in $(0,1)$.
We see this sharpening in \S\ref{subsec:Critical-remarks-SL3}; it
is also the reason the $GL(n)$ \emph{cuspidal} Eisenstein series
of \S\ref{subsec:GLn-poles} had poles only along $\langle\Lambda,\alpha_{ij}^{\vee}\rangle=\mu=1$,
and never $\mu\in(0,1)$. 

\bigskip{}

We take for granted facts peripheral to our re-examination --- in
particular, the discrete decomposition of cuspforms. We pay closer
attention to the first example of a \emph{cuspidal-data pseudo-Eisenstein
series}; the spectral expansion of these, for general reductive groups,
is the subject of Chapter VII of \cite{Langlands-Spectral}. 

\subsection{The space to be spectrally decomposed}

Let

\[
P=\left\{ \begin{pmatrix}* & *\\
0 & *
\end{pmatrix}\right\} \subset G
\]
Define the \emph{Pseudo-Eisenstein series map}
\[
\Psi_{P}:C_{c}^{\infty}(0,\infty)\to C_{c}^{\infty}(\Gamma\backslash\mathfrak{H}),\quad\Psi_{P}(\eta)(z)=\sum_{\gamma\in\Gamma\cap P\backslash\Gamma}\eta(\Im(\gamma z))
\]
 Unlike the Eisenstein series, $\Psi_{P}(\eta)$ does \emph{not}
satisfy a $\Delta$-eigenfunction equation.

Let $\mathscr{D}(0,\infty)^{*}$ be the space of distributions on
$(0,\infty)$ with measure $\frac{dy}{y^{2}}$. The constant term
along $P$ is defined via the \emph{adjoint formula }
\[
C_{P}:L^{2}(\Gamma\backslash\mathfrak{H})\to\mathscr{D}(0,\infty)^{*},
\]
\[
\langle C_{P}f,\eta\rangle_{(0,\infty)}=\langle f,\Psi_{P}(\eta)\rangle_{\Gamma\backslash\mathfrak{H}},\quad\text{for }\eta\in C_{c}^{\infty}(0,\infty)
\]
The map $C_{P}$ is continuous on $L^{2}(\Gamma\backslash\mathfrak{H})$. 

Setting
\[
L_{G}^{2}:=\ker C_{P}\quad\text{and}\quad L_{P}^{2}=\overline{\mathrm{Im}(\Psi_{P})},
\]
the adjoint formula gives the decomposition 
\[
L^{2}(\Gamma\backslash\mathfrak{H})=L_{G}^{2}\oplus L_{P}^{2}
\]

The following is a non-trivial theorem \cite{Godement_cuspforms}
\begin{thm}
\label{thm:disc-decomp} $\Delta$ has a purely discrete spectrum
on $L_{G}^{2}$: there exists an orthonormal basis of $\Delta$-eigenfunctions
for $L_{G}^{2}$. 
\end{thm}

Granting theorem \ref{thm:disc-decomp}, the spectral decomposition
of $L^{2}(\Gamma\backslash\mathfrak{H})$ reduces to decomposing 
\[
L_{P}^{2}=\overline{\left\{ \Psi_{P}(\eta):\eta\in C_{c}^{\infty}(0,\infty)\right\} }
\]
This is the simplest example of a space spanned by a cuspidal-data
pseudo-Eisenstein series. Here the cuspidal datum is the trivial representation
and is not overtly visible. Decomposition of such spaces, in great
generality, is the problem addressed in Chapter VII of Langlands \cite{Langlands-Spectral}. 

\subsection{The Eisenstein series}

Setting $\eta(\Im z)=\Im(z)^{\frac{1+s}{2}}$ formally in
the definition of $\Psi_{P}(\eta)$ recovers the \emph{Eisenstein
series} of the introduction:
\[
E_{P}(s;z)=\sum_{\gamma\in\Gamma\cap P\backslash\Gamma}\Im(\gamma z)^{\frac{1+s}{2}},\quad\Re(s)>1.
\]
We write 
\[
E(s):=E_{P}(s;\bullet)\in C^{\infty}(\Gamma\backslash\mathfrak{H}).
\]
The identity
\[
\Delta\left(\Im(z)^{\frac{1+s}{2}}\right)=\frac{1-s^{2}}{4}\cdot\Im(z)^{\frac{1+s}{2}},\qquad\lambda(s):=\frac{1-s^{2}}{4}
\]
yields 
\[
\Delta E(s)=\lambda(s)E(s)
\]
The function $E(s)$ has a meromorphic continuation to $\mathbb{C}$
and satisfies the functional equation 
\[
E(s)=c(s)E(-s),\quad c(s)=\frac{\xi(s)}{\xi(1+s)}
\]
where $\xi$ is the completed Riemann zeta function. The spectral
condition $\lambda(s)\in[0,\infty)$ holds if and only if $s\in i\mathbb{R}\cup[-1,1]$.
\begin{rem*}
$E(s)$ is treated throughout as a $C^{\infty}(\Gamma\backslash\mathfrak{H})$-valued
meromorphic function. The vanishing of $E(s_{0})$ means $E(s_{0};\bullet)\equiv0$
on $\Gamma\backslash\mathfrak{H}$. 
\end{rem*}

\subsection{\label{subsec:cont}Contour deformation}

For $\eta\in C_{c}^{\infty}(0,\infty)$, the \emph{Mellin transform
}
\[
\widehat{\eta}(s)=\int_{0}^{\infty}\eta(y)y^{-\frac{1+s}{2}}\frac{dy}{y}
\]
is entire of rapid decay on vertical lines, and \emph{Mellin inversion}
gives, for $\sigma\in\mathbb{R}$, 
\[
\eta(y)=\frac{1}{2\pi i}\int_{\sigma+i\mathbb{R}}\widehat{\eta}(s)y^{\frac{1+s}{2}}ds.
\]
Substituting into the definition of $\Psi(\eta):=\Psi_{P}(\eta)$
and exchanging sum and integral (valid for $\sigma>1$): 
\[
\Psi(\eta;z)=\sum_{\gamma\in\Gamma\cap P\backslash\Gamma}\eta(\Im(\gamma z))=\frac{1}{2\pi i}\int_{\sigma+i\mathbb{R}}\widehat{\eta}(s)E(s;z)ds
\]

\emph{This is not yet a spectral expansion}. The condition $\lambda(s)\ge0$
requires a contour shift from $\sigma>1$ to $\sigma=0$, picking
up residues at the poles of $E(s)$ on $(0,1]$: 
\[
\Psi(\eta;z)=\frac{1}{2\pi i}\int_{i\mathbb{R}}\widehat{\eta}(s)E(s;z)ds+\sum_{s_{0}\in(0,1]}\mathsf{Res}_{s=s_{0}}\left(\widehat{\eta}(s)\cdot E(s;z)\right)
\]
The line integral is the \emph{continuous part. }

A direct calculation (\cite{Garrett1}, 1.12) using the functional
equation of $E(s)$ and $|c(s)|=1$ for $s\in i\mathbb{R}$ gives
\[
\widehat{\eta}(s)=\frac{1}{2!}\langle\Psi_{P}(\eta),E(s;\bullet)\rangle\qquad\text{for }s\in i\mathbb{R},
\]
so the line integral can be rewritten as
\[
\frac{1}{2\pi i}\int_{i\mathbb{R}}\widehat{\eta}(s)E(s;z)ds=\frac{1}{4\pi i}\int_{i\mathbb{R}}\langle\Psi,E(s;\bullet)\rangle E(s;z)ds
\]
We next turn to the contribution from the residues. 

\subsection{\label{subsec:Why-simple poles}Why poles must be simple}

Suppose $E(s)$ has a pole of order $N\ge1$ at some $s_{0}\in(0,1]$,
with Laurent expansion
\[
E(s;\bullet)=\frac{E_{-N}(s_{0};\bullet)}{(s-s_{0})^{N}}+\cdots+\frac{E_{-1}(s_{0};\bullet)}{(s-s_{0})}+E_{0}(s_{0};\bullet)+\cdots
\]
The residue at $s_{0}$ picked up in the contour shift is
\[
\mathsf{Res}_{s=s_{0}}\left(\widehat{\eta}(s)E(s;\bullet)\right)=\sum_{n=1}^{N}\frac{\widehat{\eta}^{(n-1)}(s_{0})}{(n-1)!}E_{-n}(s_{0};\bullet)
\]
a linear combination of \emph{all} the negative Laurent coefficients
$E_{-1}(s_{0}),E_{-2}(s_{0}),\dots,E_{-N}(s_{0})$. 

At points of holomorphy, differentiating $\Delta E(s)=\lambda(s)E(s)$
in $s$ gives
\[
\left(\Delta-\lambda(s)\right)\left(\partial_{s}E(s)\right)=\lambda'(s)E(s)\implies\left(\Delta-\lambda(s)\right)^{2}\left(\partial_{s}E(s)\right)=0
\]
Thus, $\partial_{s}E(s)$ is \emph{not}\textbf{ }an eigenvector unless
$E(s)=0$ or $\lambda'(s)=0$ (neither holds for $s\neq0$). Similarly,
one finds that at a pole $s_{0}$ of order $N$ only the leading coefficient
$E_{-N}(s_{0})$ satisfies the $\Delta$-eigenfunction differential
equation 
\[
\Delta E_{-N}(s_{0})=\lambda(s_{0})E_{-N}(s_{0});
\]
the lower coefficients $E_{-n}$ for $n<N$ satisfy 
\[
\left(\Delta-\lambda(s_{0})\right)^{N-n+1}E_{-n}(s_{0})=0,
\]
but not the eigenfunction equation itself. 

The spectral expansion of the self-adjoint operator $\Delta$ on the
Hilbert space $L^{2}(\Gamma\backslash\mathfrak{H})$ admits only functions
satisfying the genuine eigenfunction differential equation. Hence the
poles of $E(s)$ in the interval $s_{0}\in(0,1]$ must be simple.
We emphasize that this is a heuristic argument; the proof of simplicity
uses the Maass--Selberg relation.

\begin{rem}
The argument above is the rank-one shadow of a general principle:
only the leading Laurent coefficient at a pole contributes to a spectral
expansion. In several complex variables, the role of the leading coefficient
is played by a regularization of the Eisenstein series at the distinguished
point. Sections \ref{sec:SL3} and \ref{sec:G2} carry this out in
concrete examples, where the regularization is sensitive to both the
zeros and the poles of the corresponding Eisenstein series.
\end{rem}

\subsection{Spectral decomposition: conclusion}

For $\Re(s)>1$, the only pole of $E(s)$ is at $s=1$, where $\mathsf{Res}_{s=1}E(s;z)$
is a nonzero constant in $z$. Combining the discrete decomposition
of cuspforms, the residue of $E(s)$ at $s=1$, and the continuous
spectrum from \ref{subsec:cont} yields the full spectral expansion:
\begin{thm}
For $f\in C_{c}^{\infty}(\Gamma\backslash\mathfrak{H})$, 
\[
f=\sum_{j}\langle f,\varphi_{j}\rangle\varphi_{j}+\langle f,1\rangle\frac{1}{\langle1,1\rangle}+\frac{1}{4\pi i}\int_{i\mathbb{R}}\langle f,E(s)\rangle E(s)ds,
\]
where $\{\varphi_{j}\}$ is an orthonormal basis of $\Delta$-eigenfunctions
for $L_{G}^{2}$. 
\end{thm}

\newpage{}

\section{\label{sec:SL3}Langlands' Boulder example: $SL(3)$}

In \S\ref{sec:SL2} the spectral decomposition was carried out by deforming
a one-dimensional contour. The poles of $E(s)$ in the right half-plane
are simple, all lie on the real axis, and each contributes a single
point residue to the discrete spectrum. In rank two the analogous
problem is qualitatively different. The initial contour
is now a $2$-dimensional \emph{real} integral in a $2$-dimensional
complex space, and its deformation toward the origin proceeds in two
stages (see figure \ref{fig:SL3}): 
\begin{itemize}
\item First, the two-dimensional contour is pushed across polar hyperplanes,
producing a sum of several one-dimensional integrals of residues; 
\item Next, each of these one-dimensional integrals is itself deformed and
may cross further poles, producing point residues.
\end{itemize}
This is the situation Langlands sketched for $SL(3)$ with $P=B$ the
Borel subgroup, in \S10 of Boulder notes \cite{Langlands_Boulder}.
We rework the calculation here, with the following goals.
\begin{itemize}
\item Carry out higher-rank contour deformation in the simplest setting.
The two-stage deformation (Figure \ref{fig:SL3}) is also the template
for the $G_{2}$ calculation \S\ref{sec:G2}. 
\item Explain Langlands' ``rather turbid Lemma 7.4'' of \cite{Langlands-Spectral}.
The lemma manifests here as a $3\times3$ matrix of \emph{rank-one}
and is related to the functional equations satisfied by residues of
the Borel Eisenstein series. 
\item Labesse \cite{Labesse_book}, the most recent expositor, writes: ``Moreover,
as will be seen in the $GL(3)$ example ..., many parasitic residues
do cancel. A direct argument, like the one we shall use in examples,
would not be tractable in the general case and a subtle induction
process is necessary.'' We shall see that there is nothing parasitic
about what is happening, once the \emph{zeros} of Eisenstein series
are taken into account.
\item In \S\ref{subsec:Critical-remarks-SL3} we observe that justifying
the $SL(3)$ contour deformation with the required analytic estimates
rests on a cancellation in the constant-term coefficients which, when
read carefully, sharply constrains the location of the poles of the
$SL(2)$ Eisenstein series $E(s)$ to $s=1$. This basic fact seems
to have been overlooked in the literature.
\end{itemize}

\subsection{Notation and Setup}

Let $G=SL(3,\mathbb{R})$, $\Gamma=SL(3,\mathbb{Z})$, $K=SO(3,\mathbb{R})$,
and $B=U\rtimes T$ be the Borel subgroup of upper-triangular matrices,
with $U$ its unipotent radical and $T$ its diagonal torus. Let

\[
A=\left\{ \text{diag}(a_{1},a_{2},a_{3}):a_{i}>0,\ a_{1}a_{2}a_{3}=1\right\} ,
\]
\[
T^{1}=\left\{ \text{diag}(\varepsilon_{1},\varepsilon_{2},\varepsilon_{3}):\varepsilon_{i}\in\{\pm1\},\ \varepsilon_{1}\varepsilon_{2}\varepsilon_{3}=1\right\} ,
\]
so that
\[
T=T^{1}\times A.
\]
Let
\[
\mathfrak{a}=\left\{ (x,y,z)\in\mathbb{R}^{3}:x+y+z=0\right\} ,
\]
\[
\log:A\to\mathfrak{a},\quad\text{diag}(a_{1},a_{2},a_{3})\mapsto(\log a_{1},\log a_{2},\log a_{3}).
\]
In the Langlands decomposition $G=UT^{1}AK$, the map 
\[
\mathsf{a}:G\to A,\quad\mathsf{a}(g=utak)=a
\]
is well-defined; we set 
\[
H:G\to\mathfrak{a},\quad H(g=utak)=\log(a).
\]

\subsection{Roots, coroots, and weights}

For $\mathbf{t}=\text{diag}(t_{1},t_{2},t_{3})\in T$, the positive
roots are
\[
\alpha(\mathbf{t})=t_{1}/t_{2},\quad\beta(\mathbf{t})=t_{2}/t_{3},\quad\gamma(\mathbf{t})=\alpha(\mathbf{t})\beta(\mathbf{t})=t_{1}/t_{3},
\]
with the set of simple roots $\alpha,\beta$. The corresponding coroots
are the one-parameter subgroups
\[
\alpha^{\vee}(t)=\text{diag}(t,t^{-1},1),\ \beta^{\vee}(t)=\text{diag}(1,t,t^{-1}),\ \gamma^{\vee}(t)=\text{diag}(t,1,t^{-1}).
\]
The pairing $\langle\delta_{1},\delta_{2}^{\vee}\rangle$ is defined
by $\delta_{1}\circ\delta_{2}^{\vee}(t)=t^{\langle\delta_{1},\delta_{2}^{\vee}\rangle}$,
giving
\[
\langle\delta,\delta^{\vee}\rangle=2\quad\text{for all positive roots }\delta>0,
\]
and 
\[
\langle\alpha,\beta^{\vee}\rangle=\langle\beta,\alpha^{\vee}\rangle=-1.
\]
The fundamental weights 
\[
\varpi_{\alpha}(\mathbf{t})=t_{1},\quad\varpi_{\beta}(\mathbf{t})=t_{1}t_{2}
\]
form a basis dual to $\left(\alpha^{\vee},\beta^{\vee}\right)$. Using
additive notation for roots, the half-sum of positive roots is 
\[
\rho=\frac{1}{2}(\alpha+\beta+\gamma)=\varpi_{\alpha}+\varpi_{\beta}.
\]
We choose the measure $d\Lambda$ on $\check{\mathfrak{a}}$ so that
the lattice generated by the fundamental weights is of covolume $1$;
in coordinates 
\[
\Lambda=s_{\alpha}\varpi_{\alpha}+s_{\beta}\varpi_{\beta}.
\]

The weyl group $W=S_{3}$ acts on $\mathfrak{a}$, and hence on $\check{\mathfrak{a}}$
and $\check{\mathfrak{a}}_{\mathbb{C}}:=\check{\mathfrak{a}}\otimes_{\mathbb{R}}\mathbb{C}$. 

\subsection{The space to be spectrally decomposed }

In direct analogy with \S\ref{sec:SL2}, define the pseudo-Eisenstein
series map
\[
\Psi_{B}:C_{c}^{\infty}(A)\to C_{c}^{\infty}(\Gamma\backslash G)^{K},\quad\Psi_{B}(\eta)=\sum_{\gamma\in B\cap\Gamma\backslash\Gamma}\eta\left(\mathsf{a}(\gamma g)\right)
\]
Let 
\[
L_{B}^{2}=\overline{\left\{ \Psi_{B}(\eta):\eta\in C_{c}^{\infty}(A)\right\} }\subset L^{2}(\Gamma\backslash G)^{K}
\]

\subsubsection*{Adjoint relation}

We have the adjoint formula
\[
\langle\Psi_{B}(\eta),f\rangle_{\Gamma\backslash G/K}=\langle\eta,C_{B}f\rangle_{U\backslash G/K},
\]
where the \emph{constant term }$C_{B}$ along $B$ is 
\[
C_{B}(f)(g)=\int_{\Gamma\cap U\backslash U}f(ug)du
\]
Note that $U\backslash G/K\approx A$, but the measure on $A$ is
not its Haar measure (in rank-one, it was $dy/y^{2}$ and not $dy/y$). 

\subsection{Mellin transform and the integral representation for $\Psi_{B}(\eta)$}

(Quasi-)characters on $A$ are parameterized by $\check{\mathfrak{a}}_{\mathbb{C}}$:
for $\Lambda\in\check{\mathfrak{a}}_{\mathbb{C}}$,
\[
\chi_{\Lambda}:A\to\mathbb{C}^{\times},\quad\chi(a)=e^{\langle\Lambda,H(a)\rangle}.
\]

The Mellin transform 
\[
\widehat{\eta}(\Lambda)=\int_{A}\eta(a)e^{\langle\rho-\Lambda,H(a)\rangle}da,
\]
of $\eta$ is an entire function on $\check{\mathfrak{a}}_{\mathbb{C}}$,
and it is of rapid decay along imaginary planes $\sigma+i\check{\mathfrak{a}}$
for any $\sigma\in\check{\mathfrak{a}}$. The \emph{Mellin inversion}
formula is 
\[
\eta(a)=\frac{1}{(2\pi i)^{\dim\mathfrak{a}}}\int_{\sigma+i\check{\mathfrak{a}}}\widehat{\eta}(\Lambda)e^{\langle\rho+\Lambda,H(a)\rangle}d\Lambda
\]
for any $\sigma\in\check{\mathfrak{a}}$. Cauchy's theorem implies
that the integral is independent of $\sigma.$

The Borel Eisenstein series
\[
E_{B}(\Lambda;g)=\sum_{\gamma\in B\cap\Gamma\backslash\Gamma}e^{\langle\rho+\Lambda,H(\gamma g)\rangle},
\]
converges absolutely for $\Lambda$ in the tube
\[
\left\{ \Lambda\in\check{\mathfrak{a}}_{\mathbb{C}}:\langle\Re(\Lambda),\alpha^{\vee}\rangle>1\text{ and }\langle\Re(\Lambda),\beta^{\vee}\rangle>1\right\} .
\]
Using the Mellin inversion formula, we have 
\begin{equation}
\Psi_{B}(\eta;g)=\frac{1}{(2\pi i)^{\dim\mathfrak{a}}}\int_{\sigma+i\check{\mathfrak{a}}}\widehat{\eta}(\Lambda)E_{B}(\Lambda;g)d\Lambda\label{eq:int-rep}
\end{equation}
for $\sigma\in\check{\mathfrak{a}}$ such that the Eisenstein
series converges absolutely. 

\subsection{Properties of $E_{B}$}

As a $C^{\infty}(\Gamma\backslash G)^{K}$-valued function of $\Lambda$,
$E_{B}$ admits a meromorphic continuation to $\check{\mathfrak{a}}_{\mathbb{C}}$:
\[
E_{B}:\check{\mathfrak{a}}_{\mathbb{C}}\to C^{\infty}(\Gamma\backslash G)^{K},\quad\Lambda\mapsto E_{B}(\Lambda)
\]
It has the following properties: 
\begin{itemize}
\item \emph{Functional equations.}\textbf{ }For each $w\in W$, 
\begin{equation}
E_{B}(\Lambda)=c(w,\Lambda)E_{B}(w\Lambda),\quad c(w,\Lambda)=\prod_{\delta>0:w\delta<0}\frac{\xi(\langle\Lambda,\delta^{\vee}\rangle)}{\xi(1+\langle\Lambda,\delta^{\vee}\rangle)}\label{eq:fun-eq}
\end{equation}
where $\xi$ is the completed Riemann zeta function. We have the cocycle
relation: 
\[
c(ww',\Lambda)=c(w,w'\Lambda)\cdot c(w',\Lambda).
\]
\item \emph{Constant term along $B$} is 
\begin{equation}
C_{B}E_{B}(\Lambda;g):=\int_{U\cap\Gamma\backslash U}E_{B}(ug)du=\sum_{w\in W}c(w,\Lambda)e^{\langle\rho+w\Lambda,H(g)\rangle}.\label{eq:const-term}
\end{equation}
\emph{ }
\item \emph{Regularization}\textbf{ $E_{B}^{*}$. }Let 
\[
D_{B}(\Lambda)=\prod_{\delta>0}\frac{\langle\Lambda,\delta^{\vee}\rangle}{\langle\Lambda,\delta^{\vee}\rangle-1}
\]
and define $E_{B}^{*}$ by the relation
\[
E_{B}(\Lambda)=:D_{B}(\Lambda)E_{B}^{*}(\Lambda)
\]
The function $E_{B}^{*}$ is never zero and is holomorphic on the
closed positive tube 
\[
T_{B}=\left\{ \Lambda\in\check{\mathfrak{a}}_{B,\mathbb{C}}:\Re(\langle\Lambda,\delta^{\vee}\rangle)\ge0\quad\forall\ \delta>0\right\} .
\]
Equivalently, $D_{B}$ exhibits all the zeros of $E_{B}$ and those
poles of $E_{B}$ that intersect $T_{B}$ (see \cite{Dattu-thesis}). 
\end{itemize}
For a positive root $\delta>0$, write
\[
S_{\delta}:=\left\{ \Lambda\in\check{\mathfrak{a}}_{\mathbb{C}}:\langle\Lambda,\delta^{\vee}\rangle=1\right\} 
\]
for the polar hyperplane and 
\[
H_{\delta}:=\left\{ \Lambda\in\check{\mathfrak{a}}_{\mathbb{C}}:\langle\Lambda,\delta^{\vee}\rangle=0\right\} 
\]
for the corresponding zero hyperplane. Note that $S_{\delta}=\delta/2+H_{\delta}$. 
\begin{rem}
The poles of $E_{B}^{*}$ outside $T_{B}$ are related
to the Riemann hypothesis, just as in the rank-one case. We will return
to this important issue in \S\ref{subsec:Critical-remarks-SL3}.
\end{rem}

\subsection{The Contour deformation }

The deformation of the contour 
\[
\Psi_{B}(\eta;g)=\frac{1}{(2\pi i)^{\dim\mathfrak{a}}}\int_{\sigma_{0}+i\check{\mathfrak{a}}}\widehat{\eta}(\Lambda)E_{B}(\Lambda;g)d\Lambda
\]
towards the origin is illustrated in Figure \ref{fig:SL3} and proceeds
in two stages.

\begin{figure}
\includegraphics[scale=0.25]{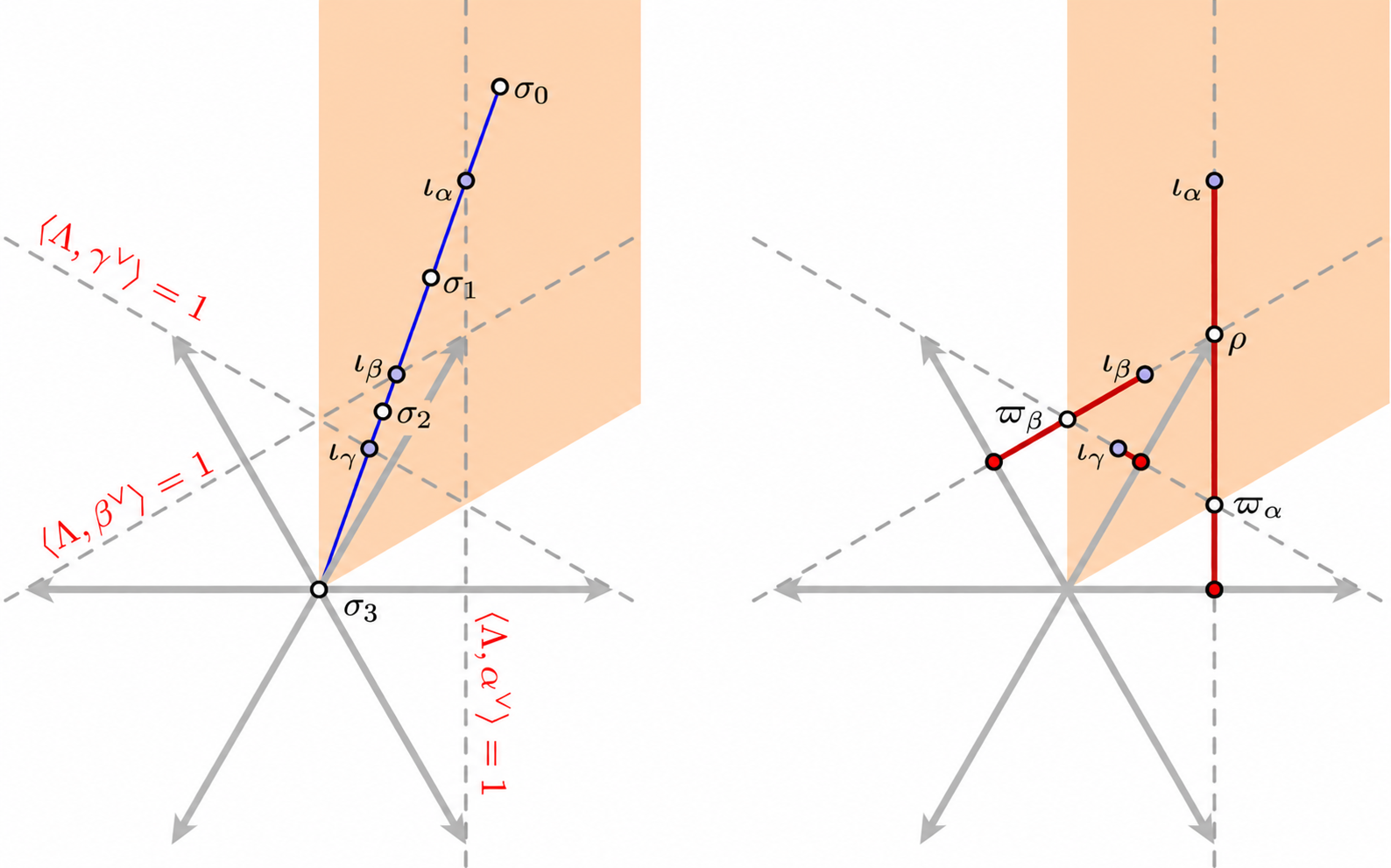}

\caption{\label{fig:SL3}Two-stage contour deformation for $SL(3)$ (courtesy
of Bill Casselman). \emph{Left:} the two-dimensional contour is pushed
from $\sigma_{0}$ along the blue path to $\sigma_{3}=0$, crossing
each polar hyperplane $S_{\delta}$ at $\iota_{\delta}$. \emph{Right:}
each one-dimensional residue integral is deformed (red paths) toward
the real point $\delta/2\in S_{\delta}$, possibly crossing the point
poles $\rho=S_{\alpha}\cap S_{\beta}$, $\varpi_{\alpha}=S_{\alpha}\cap S_{\gamma}$,
$\varpi_{\beta}=S_{\beta}\cap S_{\gamma}$. }
\end{figure}

\begin{itemize}
\item \emph{Stage I. }Push the two-dimensional contour from $\sigma_{0}$
along the blue path in Figure \ref{fig:SL3} to the origin $\sigma_{3}=0$.
This path crosses each polar hyperplane $S_{\delta}$ for $\delta>0$
once, at a point $\iota_{\delta}$. By Cauchy's theorem,
\[
\Psi_{B}(\eta;g)=\underbrace{\frac{1}{(2\pi i)^{\dim\mathfrak{a}}}\int_{i\check{\mathfrak{a}}}\widehat{\eta}(\Lambda)E_{B}(\Lambda;g)d\Lambda}_{\text{2D-continuous spectrum}}+\sum_{\delta>0}I(\iota_{\delta},\eta),
\]
where $I(\iota_{\delta},\eta)$ is the residue along $S_{\delta}$,
a one-dimensional contour integral on $S_{\delta}\cap\left\{ \iota_{\delta}+i\check{\mathfrak{a}}\right\} $. 
\item \emph{Stage II. }For each $\delta>0$, deform $I(\iota_{\delta},\eta)$
along the red path in the right-hand panel of Figure \ref{fig:SL3}
from its position over $\iota_{\delta}$ down to a one-dimensional
contour $I(\delta/2,\eta)$ over the real point $\delta/2$. The deformation
may cross point poles of the integrand: only the points
\[
\rho=S_{\alpha}\cap S_{\beta},\quad\varpi_{\alpha}=S_{\alpha}\cap S_{\gamma},\quad\varpi_{\beta}=S_{\beta}\cap S_{\gamma}
\]
are visible from this picture as candidate point-spectrum contributors.
But see \S\ref{subsec:Critical-remarks-SL3}.
\end{itemize}
The two stages decompose $\Psi_{B}(\eta)$ as
\[
\Psi_{B}(\eta)=\mathcal{C}_{2}(\eta)+\mathcal{C}_{1}(\eta)+\mathcal{C}_{0}(\eta),
\]
where $\mathcal{C}_{r}$ is the $r$-dimensional spectrum. We treat
$\mathcal{C}_{2}$ in \S\ref{subsec:Stage-I}. 

The most interesting new phenomena in this case come from $\mathcal{C}_{1}$
and are treated in \S\ref{subsec:Stage-II}. 

The discrete spectrum contribution $\mathcal{C}_{0}$ is treated in
\S\ref{subsec:parasite-SL3}, where the punchline is that two of the
three putative point contributions vanish!

\subsection{\label{subsec:Stage-I}Stage I: the two-dimensional spectrum}

For $\Lambda\in i\check{\mathfrak{a}}$, we have
\[
|c(w,\Lambda)|=1\implies\overline{c(w,\Lambda)}=c(w,\Lambda)^{-1}
\]
This fails if $\Re(\Lambda)\neq0$. From the adjoint formula and the constant
term formula for $C_{B}E_{B}$, we have ($\Lambda\in i\check{\mathfrak{a}}$)
\[
\langle\Psi_{B}(\eta),E_{B}(\Lambda)\rangle=\langle\eta,C_{B}E_{B}(\Lambda)\rangle
\]
\[
=\sum_{w\in W}\widehat{\eta}(w\Lambda)\overline{c}(w,\Lambda)=\sum_{w\in W}\widehat{\eta}(w\Lambda)c(w,\Lambda)^{-1}.
\]

\begin{prop}
\label{prop:Fourier-coefficient-2D}For $\eta\in C_{c}^{\infty}(A)$,
\[
\mathcal{C}_{2}(\eta)=\frac{1}{(2\pi i)^{2}}\int_{i\check{\mathfrak{a}}}\widehat{\eta}(\Lambda)E_{B}(\Lambda)d\Lambda
\]
\[
=\frac{1}{|W|}\frac{1}{(2\pi i)^{\dim\mathfrak{a}}}\int_{i\check{\mathfrak{a}}}\langle\Psi_{B}(\eta),E_{B}(\Lambda)\rangle E_{B}(\Lambda)d\Lambda.
\]
\end{prop}

\begin{proof}
Change the variable $\Lambda\to w\Lambda$ in the integral and use
$E_{B}(w\Lambda)=c(w,\Lambda)^{-1}E_{B}(\Lambda)$: 
\[
\frac{1}{(2\pi i)^{2}}\int_{i\check{\mathfrak{a}}}\widehat{\eta}(\Lambda)E_{B}(\Lambda)d\Lambda=\frac{1}{(2\pi i)^{2}}\int_{i\check{\mathfrak{a}}}\widehat{\eta}(w\Lambda)E_{B}(w\Lambda)d\Lambda
\]
Averaging over $W$ gives the claim. 
\end{proof}
\begin{rem}
Note that we identified $\widehat{\eta}$ as a multiple of $\langle\Psi_{B}(\eta),E_{B}\rangle$
only on the imaginary plane $i\check{\mathfrak{a}}$. This is why
the contour must be deformed all the way to $i\check{\mathfrak{a}}$
before a \emph{spectral expansion} appears. The same phenomenon occurred
at rank one in \S\ref{sec:SL2}.
\end{rem}

\subsection{On residue calculations}

Before we evaluate the one-dimensional residues, we record two preliminaries.

\subsubsection*{(a) The decomposition of measure}

For a positive root $\delta$, choose $\nu$ in $\check{\mathfrak{a}}$
satisfying $\langle\nu,\delta^{\vee}\rangle=1$. We have 
\[
\check{\mathfrak{a}}=H_{\delta}\oplus\mathbb{R}\nu\quad\text{and}\quad d\Lambda=d\Lambda_{H}dx
\]
for a unique $d\Lambda_{H}$ (independent of $\nu$), where $x=\langle\Lambda,\delta^{\vee}\rangle$.
We parameterize 
\[
S_{\delta}=\left\{ \frac{\delta}{2}+sv_{\delta}:s\in\mathbb{C}\right\} ,
\]
choosing 
\begin{equation}
v_{\alpha}=\varpi_{\beta},\quad v_{\beta}=\varpi_{\alpha},\quad v_{\gamma}=\varpi_{\alpha}-\varpi_{\beta}.\label{eq:parameterization-of-planes}
\end{equation}
A short check confirms that $\langle v_{\delta},\delta^{\vee}\rangle=0$
so that $H_{\delta}=\left\{ sv_{\delta}:s\in\mathbb{C}\right\} $. 

\subsubsection*{(b) Residue along a polar hyperplane.}

If a meromorphic function $F$ on $\check{\mathfrak{a}}_{\mathbb{C}}$
has at most a simple pole along $S_{\delta}$, then 
\[
\mathsf{Res}_{S_{\delta}}F:=\left(\langle\Lambda,\delta^{\vee}\rangle-1\right)F\big\vert_{S_{\delta}}
\]
is a meromorphic function on $S_{\delta}$. The contribution from
crossing $S_{\delta}$ in stage I above is
\[
\frac{1}{2\pi i}\int_{\sigma_{\delta}+i\mathbb{R}}\mathsf{Res}_{S_{\delta}}(\widehat{\eta}\cdot E_{B})\left(\frac{\delta}{2}+sv_{\delta}\right)ds,
\]
where $\sigma_{\delta}\in\mathbb{R}$ is characterized by $\iota_{\delta}=\delta/2+\sigma_{\delta}v_{\delta}$. 

\subsection{\label{subsec:Stage-II}Stage II: the one-dimensional spectrum}

The one-dimensional contribution is
\begin{equation}
\mathcal{C}_{1}(\eta)=\sum_{\delta>0}\frac{1}{2\pi i}\int_{\text{Re}(s)=0}\left(\widehat{\eta}\cdot\mathsf{Res}_{S_{\delta}}E_{B}\right)\left(\frac{\delta}{2}+sv_{\delta}\right)\cdot ds\label{eq:one-D}
\end{equation}
Three different residual Eisenstein series appear. We now show that
they are not independent: they are scalar multiples of one another,
with the scalars themselves residues of constant-term coefficients.

\subsubsection*{Functional equations among residues}

For $w\in W$, the functional equation $E_{B}(\Lambda)=c(w,\Lambda)E_{B}(w\Lambda)$
relates residues along different hyperplanes. Choose $w$ to map $S_{\alpha}$
onto $S_{\beta}$; rotation by $2\pi/3$ does this. The function $c(w,\Lambda)$
is holomorphic along $S_{\alpha}$ (since $\alpha$ maps to $\beta>0$)
and therefore 
\[
\left(\mathsf{Res}_{S_{\alpha}}E_{B}\right)(\alpha/2+sv_{\alpha})=\left[\mathsf{Res}_{S_{\alpha}}c(w,\bullet)E_{B}(w\cdot\bullet)\right](\alpha/2+sv_{\alpha})
\]
\[
=\left(c(w,\bullet)\cdot\mathsf{Res}_{S_{\alpha}}E_{B}(w\cdot\bullet)\right)(\alpha/2+sv_{\alpha})
\]
\[
=\frac{\xi(s-1/2)}{\xi(s+3/2)}\left(\mathsf{Res}_{S_{\beta}}E_{B}\right)(\beta/2-sv_{\beta})
\]
Similarly, taking $w$ to be the reflection defined by $\beta$, one
finds that 
\[
\left(\mathsf{Res}_{S_{\alpha}}E_{B}\right)(\alpha/2+sv_{\alpha})=\frac{\xi(s-1/2)}{\xi(s+1/2)}\left(\mathsf{Res}_{S_{\gamma}}E_{B}\right)(\gamma/2-sv_{\gamma})
\]
Using $\xi(s)=\xi(1-s)$, we get
\[
\begin{vmatrix}\frac{\xi(it-1/2)}{\xi(it+3/2)}\end{vmatrix}=1=\begin{vmatrix}\frac{\xi(it-1/2)}{\xi(it+1/2)}\end{vmatrix}\qquad(t\in\mathbb{R})
\]
just as in the rank-one case functional equation. 

\subsubsection*{Lemma 7.4 of Langlands \cite{Langlands-Spectral}}

Index the three positive roots by $i=1,2,3$ for $\alpha,\beta,\gamma$
respectively, so that $\delta_{1}=\alpha$, $\delta_{2}=\beta$, $\delta_{3}=\gamma$
and $S_{1}=S_{\delta_{1}}$ etc. Let $\sigma_{ij}\in W$ be the unique
element sending $\delta_{i}$ to $-\delta_{j}$. Each residue $\mathsf{Res}_{S_{i}}E_{B}$
has constant term along $B$ given, after taking residues termwise
in equation \ref{eq:const-term}, by
\[
\left(C_{B}\mathsf{Res}_{S_{i}}E_{B}\right)(\delta_{i}/2+sv_{i})=\frac{1}{\xi(2)}\sum_{j=1}^{3}n_{ij}(s)\exp\langle\rho+\sigma_{ij}(\delta_{i}/2+sv_{i}),H(g)\rangle.
\]
Here, 
\[
\frac{1}{\xi(2)}n_{ij}(s)=\left(\mathsf{Res}_{S_{i}}c(\sigma_{ij},\bullet)\right)(\delta_{i}+sv_{i})
\]
The $3\times3$ matrix $[n_{ij}(s)]$ has rank one due to the
functional equations above relating the residues of $E_{B}$ along
$S_{i}$. Explicitly, 
\[
[n_{ij}(s)]=\begin{pmatrix}1 & \frac{\xi\left(s-\frac{1}{2}\right)}{\xi\left(s+\frac{3}{2}\right)} & \frac{\xi\left(s+\frac{1}{2}\right)}{\xi\left(s+\frac{3}{2}\right)}\\
\frac{\xi\left(-s-\frac{1}{2}\right)}{\xi\left(-s+\frac{3}{2}\right)} & 1 & \frac{\xi\left(-s+\frac{1}{2}\right)}{\xi\left(-s+\frac{3}{2}\right)}\\
\frac{\xi\left(-s+\frac{1}{2}\right)}{\xi\left(-s+\frac{3}{2}\right)} & \frac{\xi\left(s+\frac{1}{2}\right)}{\xi\left(s+\frac{3}{2}\right)} & \frac{\xi\left(s+\frac{1}{2}\right)}{\xi\left(s+\frac{3}{2}\right)}\frac{\xi\left(-s+\frac{1}{2}\right)}{\xi\left(-s+\frac{3}{2}\right)}
\end{pmatrix}
\]
Note that 
\[
n_{ij}(s)=n_{ji}(-s)
\]

\subsubsection*{Degenerate Eisenstein series}

Let $P$ and $Q$ be respectively the $(2,1)$ and $(1,2)$ standard
maximal parabolic subgroups. For $R\in\{P,Q\}$, let
\[
E_{R}(s;g)=\sum_{\gamma\in R\cap\Gamma\backslash\Gamma}\delta_{R}^{\frac{1}{2}+\frac{s}{3}}(\gamma g),\quad\Re(s)>\frac{3}{2}
\]
Here $\delta_{R}:R\to(0,\infty)$ is the Haar measure modulus character
extended to $G$ via Iwasawa decomposition $G=RK$. The map 
\[
s\mapsto E_{R}(s;\bullet)\in C^{\infty}(\Gamma\backslash G)^{K}
\]
has a meromorphic continuation to $\mathbb{C}$. One can show that
\[
\left(\mathsf{Res}_{S_{\alpha}}E_{B}\right)(\alpha/2+sv_{\alpha})=\frac{1}{\xi(2)}E_{P}(s)
\]
\[
\left(\mathsf{Res}_{S_{\beta}}E_{B}\right)(\beta/2+sv_{\beta})=\frac{1}{\xi(2)}E_{Q}(s)
\]

\subsubsection*{Conclusion: the one-dimensional spectrum}

We can now collapse the three terms in $\mathcal{C}_{1}(\eta)$. 
\begin{prop}
For $\eta\in C_{c}^{\infty}(A)$ and $s\in i\mathbb{R}$, we have
\[
\mathcal{C}_{1}(\eta)=\frac{1}{2\pi i}\frac{1}{\xi(2)}\int_{i\mathbb{R}}\langle\Psi_{B}(\eta),E_{P}(s)\rangle E_{P}(s)ds.
\]
\end{prop}

\begin{proof}
Setting 
\[
c_{\alpha,\alpha}(s)=1,\ c_{\alpha,\beta}(s)=\frac{\xi(s-1/2)}{\xi(s+3/2)},\ c_{\alpha,\gamma}(s)=\frac{\xi(s-1/2)}{\xi(s+1/2)},
\]
a change of variable gives 
\[
\mathcal{C}_{1}(\eta)=\frac{1}{2\pi i}\int_{\text{Re}(s)=0}\left(\sum_{\delta>0}\widehat{\eta}\left(\frac{\delta}{2}+sv_{\delta}\right)c_{\alpha,\delta}^{-1}(s)\right)\mathsf{Res}_{S_{\alpha}}E_{B}\left(\frac{\alpha}{2}+sv_{\alpha}\right)ds.
\]
The adjoint formula gives
\[
\langle\Psi_{B}(\eta),\mathsf{Res}_{S_{\alpha}}E_{B}\left(\alpha/2+sv_{\alpha}\right)\rangle=\frac{1}{\xi(2)}\left(\sum_{j=1}^{3}\widehat{\eta}\left(\frac{\delta_{j}}{2}+sv_{j}\right)\overline{n_{1j}(s)}\right)
\]
Looking up the matrix $[n_{ij}(s)]$, one verifies 
\[
\overline{n_{1j}(s)}=c_{\alpha,\delta_{j}}(s)^{-1}
\]
on $s\in i\mathbb{R}$. The result follows. 
\end{proof}

\subsection{\label{subsec:parasite-SL3}The discrete spectrum: parasitic residues
unmasked}

We turn to the point-spectrum contribution $\mathcal{C}_{0}(\eta)$.
As noted in the preamble, Figure \ref{fig:SL3} suggests three contributions,
at
\[
\rho=S_{\alpha}\cap S_{\beta},\quad\varpi_{\alpha}=S_{\alpha}\cap S_{\gamma},\quad\varpi_{\beta}=S_{\beta}\cap S_{\gamma}
\]
We now show that only $\rho$ contributes; the other two yield zero.

The integrand of the deformed contour along $S_{\alpha}$ is 
\[
\mathsf{Res}_{S_{\alpha}}\left(\widehat{\eta}\cdot E_{B}\right)=\left(\widehat{\eta}\cdot E_{B}^{*}\right)\big\vert_{S_{\alpha}}\cdot\mathsf{\mathsf{Res}}_{S_{\alpha}}D_{B}
\]
We compute
\[
\left(\mathsf{Res}_{S_{\alpha}}D_{B}\right)(\alpha/2+sv_{\alpha})=(\langle\Lambda,\alpha^{\vee}\rangle-1)\prod_{\delta>0}\frac{\langle\Lambda,\delta^{\vee}\rangle}{\langle\Lambda,\delta^{\vee}\rangle-1}\big\vert_{\Lambda=\alpha/2+sv_{\alpha}}
\]
\[
=\frac{\cancel{s-\frac{1}{2}}}{s-\frac{3}{2}}\cdot\frac{s+\frac{1}{2}}{\cancel{s-\frac{1}{2}}}=\frac{s+\frac{1}{2}}{s-\frac{3}{2}}.
\]
Thus $\mathsf{Res}_{S_{\alpha}}D_{B}$ has no pole at the point $\text{\ensuremath{\varpi_{\alpha}}}=S_{\alpha}\cap S_{\gamma}$!
We summarize the above calculation in the proposition below, which
addresses the ``parasitic'' issue raised by Labesse. This already
shows the starring role of the zeros of Eisenstein series $E_{B}$. 
\begin{prop}
For any $\eta\in C_{c}^{\infty}(A)$, the function $\mathsf{Res}_{S_{\alpha}}\left(\widehat{\eta}\cdot E_{B}\right)$
is holomorphic at $\varpi_{\alpha}\in S_{\alpha}$ and the function
$\mathsf{Res}_{S_{\beta}}\left(\widehat{\eta}\cdot E_{B}\right)$
is holomorphic at $\varpi_{\beta}\in S_{\beta}$. 
\end{prop}

\emph{The contribution from $\rho$. }We compute
\[
\mathsf{Res}_{\rho}\mathsf{Res}_{S_{\alpha}}(\widehat{\eta}\cdot E_{B})=\widehat{\eta}(\rho)E_{B}^{*}(\rho)\cdot\mathsf{Res}_{s=\frac{3}{2}}\frac{s+\frac{1}{2}}{s-\frac{3}{2}}=2\widehat{\eta}(\rho)E_{B}^{*}(\rho).
\]
One computes using the explicit formula for $C_{B}E_{B}$ that
\[
2C_{B}E_{B}^{*}(\rho)=\frac{1}{\xi(2)\xi(3)}
\]
and
\[
\langle\Psi_{B}(\eta),1\rangle=\widehat{\eta}(\rho)
\]
Thus, 
\[
2\widehat{\eta}(\rho)E_{B}^{*}(\rho)=\langle\Psi_{B}(\eta),1\rangle\frac{1}{\xi(2)\xi(3)}
\]
Combining the results for $\mathcal{C}_{2}(\eta),\ \mathcal{C}_{1}(\eta)$,
and $\mathcal{C}_{0}(\eta)$, we have
\begin{thm}
For $\eta\in C_{c}^{\infty}(A)$, we have
\[
\Psi_{B}(\eta)=\langle\Psi_{B}(\eta),1\rangle\frac{1}{\xi(2)\xi(3)}+\frac{1}{2\pi i}\frac{1}{\xi(2)}\int_{i\mathbb{R}}\big\langle\Psi_{B}(\eta),E_{P}(s)\big\rangle E_{P}(s)ds
\]
\[
+\frac{1}{|W|}\frac{1}{(2\pi i)^{2}}\int_{i\check{\mathfrak{a}}_{B}}\langle\Psi_{B}(\eta),E_{B}(\Lambda)\rangle E_{B}(\Lambda)d\Lambda
\]
\end{thm}

\subsection{\label{subsec:Critical-remarks-SL3} Constraining the location of
the poles of cuspidal Eisenstein series}

The contour deformation above relies, at one technical step, on a
cancellation that we have not yet examined. We isolate it here because
it sharply restricts where the poles of cuspidal Eisenstein series
involved can lie.

To analytically justify the deformation in Stage II, one must control
$\mathsf{Res}_{S_{\alpha}}E_{B}(\Lambda)$ along the red path in Figure
\ref{fig:SL3} on $S_{\alpha}$. \emph{This red path exits the positive
cone. }In these regions, the term 
\[
c(w,\Lambda)=\prod_{\delta>0}\frac{\xi(\langle\Lambda,\delta^{\vee}\rangle)}{\xi(1+\langle\Lambda,\delta^{\vee}\rangle)}
\]
in $C_{B}E_{B}(\Lambda)$ \emph{has poles from the critical
zeros }of $\xi$. 

However, we need to understand only the residue $\mathsf{Res}_{S_{\alpha}}c(w,\Lambda)$,
where a magical cancellation happens: 
\[
\left(\mathsf{Res}_{S_{\alpha}}c(w,\bullet)\right)(\alpha/2+sv_{\alpha})=\frac{1}{\xi(2)}\frac{\xi(\langle\Lambda,\beta^{\vee}\rangle)}{\xi({\color{green}\mathbf{1}}+\langle\Lambda,\beta^{\vee}\rangle)}\frac{\xi({\color{red}\mathbf{1}}+\langle\Lambda,\beta^{\vee}\rangle)}{\xi(2+\langle\Lambda,\beta^{\vee}\rangle)}\big\vert_{\Lambda=\alpha/2+sv_{\alpha}}
\]
\[
=\frac{1}{\xi(2)}\frac{\xi(\langle\Lambda,\beta^{\vee}\rangle)}{\xi(2+\langle\Lambda,\beta^{\vee}\rangle)}\big\vert_{\Lambda=\alpha/2+sv_{\alpha}}
\]

The denominator factor\emph{ $\xi({\color{green}\mathbf{1}}+\langle\Lambda,\beta^{\vee}\rangle)$
}comes from the Langlands-Shahidi formula and is known in greater
generality,\emph{ }starting\emph{ }from Langlands' work ``Euler products''\emph{
\cite{Langlands-EulerProduct} }from 1967. 

The numerator factor\emph{ $\xi({\color{red}\mathbf{1}}+\langle\Lambda,\beta^{\vee}\rangle)$}
comes from the \textbf{\emph{location of the poles}} of an $L$-function
appearing in the Langlands-Shahidi formula. That an $L$-function
has a pole exactly at $1$, and not at any other point in $(0,1)$,
is the kind of statement that for cuspidal $L$-functions of $GL(n)$
is settled (Jacquet-Shalika), \emph{but in greater generality is fragmentary. }

It is not clear to the author that Langlands \emph{ever addresses
this issue in his Eisenstein series manuscript }\cite{Langlands-Spectral}
written some years\emph{ }\textbf{\emph{before }}his landmark work
``Euler Products'' \cite{Langlands-EulerProduct}. It is not clear
he was even aware of the seriousness of this issue in $1964$. We
note however that Langlands recognized this cancellation for the first
time in appendix III of \cite{Langlands-Spectral}. 

\newpage{}

\section{\label{sec:G2}$G_{2}$: explication of Langlands' calculation}

The contour-deformation framework set up in \S\ref{sec:SL3} for $SL(3)$
applies to any split semisimple group of rank two. Appendix
III of Langlands' manuscript \cite{Langlands-Spectral} concerns the exceptional
Chevalley group $G=G_{2}$:
\begin{quotation}
``We shall see that the subspace corresponding to the discrete spectrum,
that is, the spectrum of dimension $0$, is of dimension two, consisting
of the constant functions and another eigenspace of dimension one.''
--- Langlands \cite{Langlands-Spectral}. 
\end{quotation}
This is a long calculation where Langlands writes formulas that a
``conscientious reader will readily verify''. It is certain that
such a reader ``will be losing heart, for we still have $\mathfrak{s}_{6}$
to work through. He is urged to persist, for the final result is very
simple. I do not know the reason''. In my view, this calculation
is a virtuoso display of Langlands\textquoteright{} computational
power, comparable to the great calculations of Gauss and Euler. 

Fix $B=U\rtimes T$ a Borel subgroup of $G$. Our goal, like Langlands',
is the discrete spectrum contribution obtained by deforming
\[
\Psi_{B}(\eta)=\frac{1}{(2\pi i)^{\dim\mathfrak{a}}}\int_{\sigma+i\check{\mathfrak{a}}}\widehat{\eta}(\Lambda)E_{B}(\Lambda)d\Lambda
\]
as shown in Figure \ref{fig:G2}. We shall see that since our approach
tracks the zeros of the Eisenstein series $E_{B}$, the entire calculation
will appear as simple high-school algebra. 

We begin by recalling the notation, which is similar to the $SL(3)$
case. 

\subsection{Notation}

Let $\{\alpha,\beta\}$ be the simple roots, $\alpha$ short and $\beta$
long, with 
\[
\langle\alpha,\beta^{\vee}\rangle=-1,\quad\langle\beta,\alpha^{\vee}\rangle=-3.
\]
The six positive roots are listed in Figure \ref{fig:G2}. We label
them clockwise as $\beta_{1},\dots,\beta_{6}$, where 
\[
\beta_{1}=\beta,\quad\beta_{2}=\alpha+\beta,\quad\beta_{3}=3\alpha+2\beta\quad\beta_{4}=2\alpha+\beta\quad\beta_{5}=3\alpha+\beta,\quad\beta_{6}=\alpha.
\]
Set $\check{\mathfrak{a}}:=X^{*}(T)\otimes_{\mathbb{Z}}\mathbb{R}$,
$\check{\mathfrak{a}}_{\mathbb{C}}=\check{\mathfrak{a}}\otimes\mathbb{C}$,
and use coordinates 
\[
\Lambda=s_{\alpha}\varpi_{\alpha}+s_{\beta}\varpi_{\beta}
\]
A direct computation shows $\rho=\varpi_{\alpha}+\varpi_{\beta}$. 

We set
\[
S_{i}=S_{\beta_{i}}=\left\{ \Lambda\in\check{\mathfrak{a}}_{\mathbb{C}}:\langle\Lambda,\beta_{i}^{\vee}\rangle=1\right\} 
\]
\[
=\left\{ \beta_{i}/2+zv_{i}:z\in\mathbb{C}\right\} ,\quad\langle v_{i},\beta_{i}^{\vee}\rangle=0
\]
We use the following choices: 
\[
v_{1}=\varpi_{\alpha},\quad v_{5}=\varpi_{\beta}-\varpi_{\alpha},\quad v_{6}=\varpi_{\beta}.
\]
As in the $SL(3)$ case, we use

\[
d\Lambda=ds_{\alpha}ds_{\beta}=dzdx_{i},\quad x_{i}=\langle\Lambda,\beta_{i}^{\vee}\rangle-1
\]

\emph{Convention: residues at simple and double poles. }Suppose $f$
is a meromorphic function on $S_{i}$, parameterized as $\Lambda=\beta_{i}+zv_{i}$,
with a pole of order at most two at $z=z_{0}$. We write 
\[
f(\beta_{i}+zv_{i})=\frac{\mathsf{Lead}_{z_{0}}f}{(z-z_{0})^{2}}+\frac{\mathsf{Res}_{z_{0}}f}{(z-z_{0})}+(\text{holomorphic at }z_{0})
\]
Since the surrounding contour integral carries a $1/2\pi i$ factor,
crossing the pole during deformation contributes the residue of $\widehat{\eta}\cdot f$
at $z_{0}$, namely
\[
D_{v_{i}}\widehat{\eta}(\Lambda_{0})\cdot\mathsf{Lead}_{z_{0}}f+\widehat{\eta}(\Lambda_{0})\mathsf{Res}_{z_{0}}f,
\]
where $\Lambda_{0}=\beta_{i}+z_{0}v_{i}$ and $D_{v_{i}}$ denotes
the directional derivative along $v_{i}$.

\subsection{The setup of contour deformation }

Deforming the contour integral 
\[
\Psi_{B}(\eta)=\frac{1}{(2\pi i)^{\dim\check{\mathfrak{a}}}}\int_{\sigma+i\check{\mathfrak{a}}}\widehat{\eta}(\Lambda)\cdot E_{B}(\Lambda)d\Lambda
\]
as in \ref{fig:G2} produces 
\begin{itemize}
\item One 2D integral over the origin. 
\item Six 1D integrals corresponding to the hyperplane $S_{\gamma}=\frac{\gamma}{2}+H_{\gamma}$. 
\end{itemize}
\begin{figure}
\includegraphics[scale=0.35]{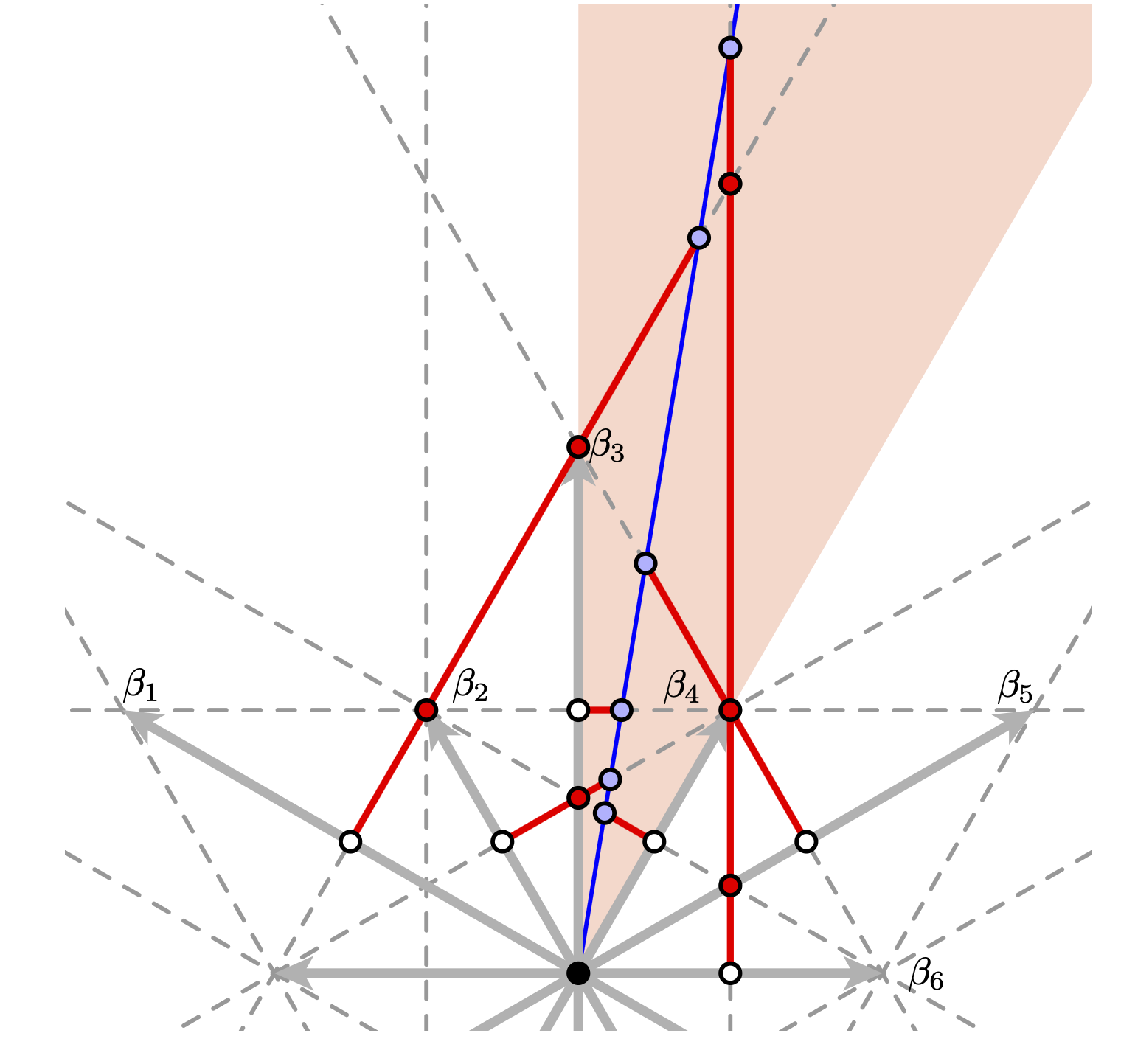}

\caption{\label{fig:G2}Contour deformation for $G_{2}$. Courtesy of Bill
Casselman.}
\end{figure}

Our goal, like Langlands', is to understand the \emph{point or discrete
spectrum} contribution. 

A glance at the figure shows that no point poles are crossed along $S_{3}$
or $S_{4}$, so these contribute nothing to the discrete spectrum.
Thus, the discrete spectrum is the sum of point contributions from
$\mathsf{Res}_{S_{i}}E_{B}$ for $i=1,2,5,6$. 

Let 
\[
D_{B}(\Lambda)=\prod_{\gamma>0}\frac{\langle\Lambda,\gamma^{\vee}\rangle}{\langle\Lambda,\gamma^{\vee}\rangle-1}
\]
and 
\[
E_{B}^{*}(\Lambda):=D_{B}(\Lambda)^{-1}\cdot E_{B}(\Lambda).
\]
Then $E_{B}^{*}$ is \emph{holomorphic} in the tube over the closed
positive cone. The function $D_{B}(\Lambda)$ captures the singularities
of the Eisenstein series $E_{B}(\Lambda)$ relevant for our calculations. 

\subsection{Along $S_{6}$}

We do this calculation in detail to show how simple matters are: the
relevant poles of $\mathsf{Res}_{S_{6}}E_{B}$ are governed by
\[
\mathsf{Res}_{S_{6}}D_{B}=\mathsf{Res}_{S_{6}}\prod_{\gamma > 0}\frac{\langle\Lambda,\gamma^{\vee}\rangle}{\langle\Lambda,\gamma^{\vee}\rangle-1}
\]
\[
=\frac{(s_{\alpha})(s_{\beta})(s_{\alpha}+s_{\beta})(s_{\alpha}+2s_{\beta})(s_{\alpha}+3s_{\beta})(2s_{\alpha}+3s_{\beta})}{(s_{\beta}-1)(s_{\alpha}+s_{\beta}-1)(s_{\alpha}+2s_{\beta}-1)(s_{\alpha}+3s_{\beta}-1)(2s_{\alpha}+3s_{\beta}-1)}\vert_{s_{\alpha}=1}
\]
\[
=\frac{(1)\cancel{(s_{\beta})}(1+s_{\beta})(1+2s_{\beta})\cancel{(1+3s_{\beta})}(2+3s_{\beta})}{(s_{\beta}-1)\cancel{(s_{\beta})}(2s_{\beta})(3s_{\beta})\cancel{(1+3s_{\beta})}}
\]
\[
=\frac{(1+s_{\beta})(1+2s_{\beta})(2+3s_{\beta})}{6(s_{\beta}-1)(s_{\beta}^{2})}
\]
Writing $\Lambda=\frac{\alpha}{2}+z\varpi_{\beta}\in S_{\alpha}$,
\[
s_{\beta}=\langle\Lambda,\beta^{\vee}\rangle=\langle\frac{\alpha}{2}+z\varpi_{\beta},\beta^{\vee}\rangle=-\frac{1}{2}+z
\]
In the new coordinates,
\[
\frac{(1+s_{\beta})(1+2s_{\beta})(2+3s_{\beta})}{6(s_{\beta}-1)(s_{\beta}^{2})}\vert_{s_{\beta}=z-\frac{1}{2}}=\frac{(z+\frac{1}{2})(2z)(3z+\frac{1}{2})}{6(z-\frac{3}{2})(z-\frac{1}{2})^{2}}
\]
The contribution to the point spectrum comes from $z=\frac{1}{2},\ \frac{3}{2}$. 

\subsubsection*{At $z=\frac{3}{2}$ or $\Lambda=\rho$}

The pole is simple. We get 
\[
\widehat{\eta}(\rho)\cdot\mathsf{Res}_{\rho}\left(\mathsf{Res}_{S_{6}}E_{B}\right)
\]

\subsubsection*{At $z=\frac{1}{2}$ or $\Lambda=\beta_{4}$}

We have a second-order pole, and the point residue obtained by contour
deformation is 
\[
D_{\varpi_{\beta}}\widehat{\eta}(\beta_{4})\cdot\mathsf{Lead}_{\beta_{4}}\left(\mathsf{Res}_{S_{6}}E_{B}\right)+\widehat{\eta}(\beta_{4})\cdot\mathsf{Res}_{\beta_{4}}\left(\mathsf{Res}_{S_{6}}E_{B}\right)
\]

\subsection{Along $S_{5}$}

We have
\[
\mathsf{Res}_{S_{5}}D_{B}\left(\beta_{5}/2+z(\varpi_{\beta}-\varpi_{\alpha})\right)=\frac{(z)(z+\frac{3}{2})}{\left(z-\frac{1}{2}\right)^{2}}
\]
The double pole at $z=\frac{1}{2}$ corresponding to $\beta_{4}$
contributes 
\[
D_{\varpi_{\beta}-\varpi_{\alpha}}\widehat{\eta}(\beta_{4})\cdot\mathsf{Lead}_{\beta_{4}}\left(\mathsf{Res}_{S_{5}}E_{B}\right)+\widehat{\eta}(\beta_{4})\cdot\mathsf{Res}_{\beta_{4}}\left(\mathsf{Res}_{S_{5}}E_{B}\right).
\]

\subsection{Along $S_{1}$ and $S_{2}$}

We have
\[
\left(\mathsf{Res}_{S_{1}}D_{B}\right)(\beta/2+z\varpi_{\alpha})=\frac{(z+\frac{3}{2})(z)}{(z-\frac{5}{2})(z-\frac{1}{2})}
\]
The point $z=\frac{5}{2}$ corresponds to $\rho$, which lies outside
the red path on $S_{1}$ and is not crossed. Only the simple pole
at $z=\frac{1}{2}$, which corresponds to $\Lambda=\beta_{2}$, contributes:

\[
\widehat{\eta}(\beta_{2})\cdot\mathsf{Res}_{\beta_{2}}\mathsf{Res}_{S_{1}}E_{B}
\]

On $S_{2}$ a similar elementary calculation shows that $\mathsf{Res}_{S_{2}}D_{B}$
is holomorphic at every point of its red path, and there is no contribution.

\subsection{Support of the discrete spectrum}

Summing the \emph{point} residues from $S_{1},\dots,S_{6}$, the discrete
spectrum part of $\Psi_{B}(\eta)$ is supported on $\{\rho,\beta_{2},\beta_{4}\}$
and equals 
\[
\widehat{\eta}(\rho)\cdot\mathsf{Res}_{\rho}\left(\mathsf{Res}_{S_{6}}E_{B}\right)
\]
\[
+\widehat{\eta}(\beta_{2})\cdot\mathsf{Res}_{\beta_{2}}\left(\mathsf{Res}_{S_{1}}E_{B}\right)
\]
\[
+D_{\varpi_{\beta}-\varpi_{\alpha}}\widehat{\eta}(\beta_{4})\cdot\mathsf{Lead}_{\beta_{4}}\left(\mathsf{Res}_{S_{5}}E_{B}\right)+\widehat{\eta}(\beta_{4})\cdot\mathsf{Res}_{\beta_{4}}\left(\mathsf{Res}_{S_{5}}E_{B}\right)
\]
\[
+D_{\varpi_{\beta}}\widehat{\eta}(\beta_{4})\cdot\mathsf{Lead}_{\beta_{4}}\left(\mathsf{Res}_{S_{6}}E_{B}\right)+\widehat{\eta}(\beta_{4})\cdot\mathsf{Res}_{\beta_{4}}\left(\mathsf{Res}_{S_{6}}E_{B}\right).
\]

\subsection{The discrete spectrum}

We now identify these six terms as scalar multiples of just two automorphic
forms.

\emph{Hecke eigenfunctions among Laurent coefficients. By the same
reasoning as in the $SL(2)$ case, only those Laurent coefficients of
$E_{B}$ at the candidate points that are spherical Hecke eigenfunctions
contribute to the discrete spectrum. Among the relevant Laurent coefficients
at $\rho,\ \beta_{2}$ and $\beta_{4}$, the only such eigenfunctions
are}
\[
E_{B}^{*}(\rho),\quad E_{B}^{*}(\beta_{2}),\quad E_{B}^{*}(\beta_{4}).
\]
The functional equations of $E_{B}$ relate 
\[
E_{B}^{*}(\beta_{2})=C\cdot E_{B}^{*}(\beta_{4}),\quad\text{for some }C\neq0.
\]

\emph{Reduction of the leading-term contributions. }The leading coefficients
at the double poles are read off directly:

\[
\mathsf{Lead}_{z=\frac{1}{2}}\frac{(z+\frac{1}{2})(2z)(3z+\frac{1}{2})}{6(z-\frac{3}{2})(z-\frac{1}{2})^{2}}=-\frac{1}{3}\text{ and }\mathsf{Lead}_{z=\frac{1}{2}}\frac{(z)(z+\frac{3}{2})}{\left(z-\frac{1}{2}\right)^{2}}=1,
\]
Therefore
\[
\mathsf{Lead}_{\beta_{4}}\mathsf{Res}_{S_{6}}E_{B}=-\frac{1}{3}E_{B}^{*}(\beta_{4}),\quad\mathsf{Lead}_{\beta_{4}}\mathsf{Res}_{S_{5}}E_{B}=E_{B}^{*}(\beta_{4}).
\]
The two leading-term contributions therefore combine to
\[
\left(D_{\varpi_{\beta}-\varpi_{\alpha}}-\frac{1}{3}D_{\varpi_{\beta}}\right)\widehat{\eta}(\beta_{4})E_{B}^{*}(\beta_{4})=\frac{1}{3}D_{\beta}\widehat{\eta}(\beta_{4})\cdot E_{B}^{*}(\beta_{4})
\]
using $\beta=-3\varpi_{\alpha}+2\varpi_{\beta}$. Langlands denotes
$D_{\beta_{1}}$ as $D_{1}$ in appendix III of \cite{Langlands-Spectral}. 

\emph{Conclusion. }Putting everything together, the discrete spectrum
is

\[
\widehat{\eta}(\rho)\cdot E_{B}^{*}(\rho)\cdot c_{\rho}+\left(a\widehat{\eta}(\beta_{2})+b\widehat{\eta}(\beta_{4})+cD_{\beta_{1}}\widehat{\eta}(\beta_{4})\right)\cdot E_{B}^{*}(\beta_{4})
\]
for some constants $a,b,c,c_{\rho}\neq0$. Since for any $\eta\in C_{c}^{\infty}(A_{B})$,
we have
\[
\langle\Psi_{B}(\eta),E_{B}^{*}(\beta_{4})\rangle=C'\left(a\widehat{\eta}(\beta_{2})+b\widehat{\eta}(\beta_{4})+cD_{\beta_{1}}\widehat{\eta}(\beta_{4})\right)
\]
for some $C'\neq0$, it follows that the exponents of $E_{B}^{*}$
are $-\beta_{2},-\beta_{4}$. Note that there is a logarithmic derivative
in the constant term $C_{B}E_{B}^{*}(\beta_{4})$. By Langlands' square-integrability
criterion, 
\[
E_{B}^{*}(\beta_{4})\in L^{2}\left(G(F)\backslash G(\mathbb{A})\right).
\]

\begin{thm}
In the above setting, the discrete spectrum is two-dimensional, spanned
by
\[
E_{B}^{*}(\rho)\quad\text{and}\quad E_{B}^{*}(\beta_{4}).
\]
\end{thm}

\begin{rem*}
A \emph{logical} proof that the contour deformation picks only Hecke
eigenfunctions is missing. The simplicity of the poles of $E(s)$
in the $SL(2)$-case followed from Maass--Selberg. The situation here
is far more complicated. 
\end{rem*}

\section{\label{sec:General}A program to rethink Langlands' work}

The $G_{2}$ calculation of \S\ref{sec:G2}, like the $SL(3)$ calculation
that preceded it, was carried out by deforming a contour that contained
information about both the poles and the zeros of the Borel Eisenstein
series explicitly. Once this is done, the residual spectrum is recovered
as a regularization of the Borel Eisenstein series at a finite list
of distinguished points, and the calculation that Langlands warned
would be unavoidably intricate reduces to elementary algebra.

The aim of this section is to extract from these examples a coherent
program. The central heuristic suggested by examples is that the \emph{support
of the discrete spectrum} due to $E_{B}(\Lambda)$ is captured by
a single rational function $D_{B}(\Lambda)$ that records \emph{both
the relevant poles and zeros} of $E_{B}(\Lambda)$. The discrete spectrum
then is spanned by the regularized Eisenstein series $E_{B}^{*}$
evaluated at these points. 

The section is in two parts. Subsection \ref{subsec:Program-Borel}
illustrates the program in the simplest case, that of the Borel Eisenstein
series, where every ingredient --- the Eisenstein series, its regularization,
the central heuristic, and the resulting automorphic forms --- can
be exhibited concretely. Subsection \ref{subsec:Program-general}
then formulates the expectation for a general parabolic subgroup,
and the section closes with remarks and a coda.

\subsection{\label{subsec:Program-Borel}The program in the simplest case: the
Borel Eisenstein series}

Throughout \ref{subsec:Program-Borel}, we work with the Borel Eisenstein
series. We fix the following notation.

\subsubsection*{Notation}
\begin{itemize}
\item $G$: a connected split semisimple group over a number field $F$. 
\item $B=U\rtimes T$: fixed Borel subgroup $B\subset G$ with a fixed maximal
$F$-split torus $T$. 
\item $\Delta_{B}$: the set of simple roots with respect to $B$.
\item $\mathcal{A}[G]$: the space of automorphic forms on $[G]=G(F)\backslash G(\mathbb{A})$. 
\item $\check{\mathfrak{a}}_{B,\mathbb{C}}:=X^{*}(T)\otimes_{\mathbb{Z}}\mathbb{C}$
\end{itemize}
\begin{rem}
When we say a $G(\mathbb{A})$-intertwining operator, we follow the
usual convention of dealing with archimedean places using the so-called
$(\mathfrak{g},K_{\infty})$-modules. 
\end{rem}

\newpage{}

\subsubsection{Borel Eisenstein series, representation theoretic}

For $\Lambda\in\check{\mathfrak{a}}_{B,\mathbb{C}}$, let
\[
I_{B}^{G}(\mathbf{1},\Lambda)=\text{Ind}_{B(\mathbb{A})}^{G(\mathbb{A})}(e^{\langle\rho+\Lambda,H_{B}(\cdot)\rangle})
\]
As usual, we identify the representation spaces of $I_{B}^{G}(\mathbf{1},\Lambda)$
with a single space 
\[
V_{\Lambda}=V\subset \mathcal{A}( T(F)U(\mathbb{A})\backslash G(\mathbb{A})),
\]
but the action of $G(\mathbb{A})$ depends on $\Lambda$. 

The Eisenstein series
\[
E_{B}(\mathbf{1}):\check{\mathfrak{a}}_{B,{\mathbb{C}}}\to\text{Hom}_{\mathbb{C}}\left(V,\mathcal{A}[G]\right),\ \ E_{B}(\mathbf{1})(\Lambda)\in\text{Hom}_{G(\mathbb{A})}\left(V_{\Lambda},\mathcal{A}[G]\right),
\]
is initially defined on an open set in $\check{\mathfrak{a}}_{B,{\mathbb{C}}}$
by
\[
E_{B}(\varphi,\Lambda):=\left(E_{B}(\mathbf{1})(\Lambda)\right)(\varphi)=\sum_{\gamma\in P(F)\backslash G(F)}e^{\langle\rho+\Lambda,H_{B}(\gamma g)\rangle}\varphi(\gamma g),
\]
and extends to a meromorphic function on $\check{\mathfrak{a}}_{\mathbb{C}}.$
If $\varphi=1$, we get 
\[
E_{B}(\Lambda):=E_{B}(1,\Lambda).
\]

\subsubsection{Regularization}

Just as before, one has
\[
E_{B}(\mathbf{1},\Lambda)=E_{B}^{*}(\mathbf{1},\Lambda)D_{B}(\mathbf{1},\Lambda),\quad D_{B}(\mathbf{1},\Lambda):=D_{B}(\Lambda)=\prod_{\gamma>0}\frac{\langle\Lambda,\gamma^{\vee}\rangle}{\langle\Lambda,\gamma^{\vee}\rangle-1}
\]
and $E_{B}^{*}(\mathbf{1},\bullet)$ is holomorphic and non-zero (as
an intertwining operator) on the tube 
\[
T_{B}=\left\{ \Lambda\in\check{\mathfrak{a}}_{\mathbb{C}}:\Re\langle\Lambda,\alpha^{\vee}\rangle\ge0\ \forall\alpha\in\Delta_{B}\right\} 
\]
over the closed positive cone. We note that $E_{B}^{*}$ is the only
way to regularize $E_{B}$ so that it is still a non-zero intertwining
operator for every $\Lambda\in T_{B}$. The other Laurent coefficients
do not give $G(\mathbb{A})$-intertwining operators. 

\subsubsection{\label{subsubsec:disc-examples}The central heuristic}

\bigskip{}

\emph{Spherical case. }As before, we need to deform the integral (for
$\eta\in C_{c}^{\infty}(A_{B})$)
\[
\Psi_{B}(\eta)=\frac{1}{(2\pi i)^{\dim\check{\mathfrak{a}}_{B}}}\int_{\sigma+i\check{\mathfrak{a}}_{B}}\widehat{\eta}(\Lambda)E_{B}(\Lambda)d\Lambda
\]
\[
=\frac{1}{(2\pi i)^{\dim\check{\mathfrak{a}}_{B}}}\int_{\sigma+i\check{\mathfrak{a}}_{B}}\widehat{\eta}(\Lambda)D_{B}(\Lambda)E_{B}^{*}(\Lambda)d\Lambda
\]
in several steps. The central heuristic suggested by our examples
is that the special points that contribute to the discrete spectrum
arise from $D_{B}(\Lambda)$. This heuristic, if true, reduces the
problem of contour deformation to the study of the Paley-Wiener distribution
\[
\eta\mapsto\frac{1}{(2\pi i)^{\dim\check{\mathfrak{a}}_{B}}}\int_{\sigma+i\check{\mathfrak{a}}_{B}}\widehat{\eta}(\Lambda)D_{B}(\Lambda)d\Lambda
\]
If $\rho_{\#}$ is a point in the closed positive cone that contributes
to the discrete spectrum, then the contribution is simply $E_{B}^{*}(\rho_{\#})$. 

\bigskip{}

\emph{Representation theoretic view}. The exact same dynamic appears
at the representation theoretic level. Now, one has to deform 
\[
\Psi_{B}(\eta,\varphi)=\frac{1}{(2\pi i)^{\dim\check{\mathfrak{a}}_{B}}}\int_{\sigma+i\check{\mathfrak{a}}_{B}}\widehat{\eta}(\Lambda)E_{B}(\varphi,\Lambda)d\Lambda
\]
for each $\varphi\in V=I_{B}^{G}(\mathbf{1},\Lambda)$ (as above),
and one looks at the image under
\[
E_{B}^{*}(\mathbf{1},\rho_{\#}):I_{B}^{G}(\mathbf{1},\rho_{\#})\to L^{2}(G(F)\backslash G(\mathbb{A}))
\]

\bigskip{}

\emph{Examples}. We now exhibit examples of discrete automorphic representations
the heuristic predicts. Let
\[
\kappa:SL(2,\mathbb{C})\to\widehat{G}
\]
be a \emph{distinguished }morphism (in the sense of unipotent orbits
theory). By restricting to the torus diagonal in $SL(2,\mathbb{C})$,
one obtains a coroot 
\[
2\rho_{\kappa}\in X_{*}(\widehat{T})\simeq X^{*}(T).
\]
We assume that $\rho_{\kappa}$ is in the closed positive cone. The
following is likely proved in \cite{Miller_Annals}\cite{Moeglin_Compositio},
in a totally different formulation. 
\begin{fact}
\label{fact:distinguished-L2}For $\kappa$ distinguished, let 
\[
V(\mathbf{1},\rho_{\kappa}):=E_{B}^{*}(\mathbf{1},\rho_{\kappa})\left(I_{B}^{G}(\mathbf{1},\rho_{\kappa})\right).
\]
The representation $V(\mathbf{1},\rho_{\kappa})$ is an irreducible
unitary spherical representation of $G(\mathbb{A})$ occurring in $L^{2}\left(G(F)\backslash G(\mathbb{A})\right)$. 

These representations were used to prove a conjecture of Arthur that
the spherical constituent of an unramified principal series of a Chevalley
group over any local field of characteristic zero is unitarizable
if its Langlands parameter coincides with half the weighted marking
of a coadjoint nilpotent orbit of the Langlands dual Lie algebra.
\end{fact}

\begin{rem}
It is highly desirable to obtain a direct proof of the above fact.
Miller's very readable paper \cite{Miller_Annals} is, unfortunately,
a computer verification of this conjecture.
\end{rem}

\subsubsection{\label{subsec:Poles-disc}Poles of some unramified Eisenstein series}

We use the notation from subsection \ref{subsec:deg-Eis-poles} and
generalize the result there. Let 
\[
\kappa:SL(2,\mathbb{C})\to\widehat{M}
\]
be a \emph{distinguished }morphism. From fact \ref{fact:distinguished-L2},
we get an irreducible representation
\[
\pi_{\kappa}\subset L^{2}(A_{M}M(F)\backslash M(\mathbb{A}))
\]
We have the Eisenstein series $E_{P}(\pi_{\kappa},s)$. The poles
of this Eisenstein series can be obtained using the method in \cite{Dattu-thesis}. 

Using the adjoint representation of $\widehat{M}$ on $\widehat{\mathfrak{n}}$
and the homomorphism $\kappa:SL(2,\mathbb{C})\to\widehat{M}$ we obtain,
as in theorem \ref{thm:deg-Eis}, the decompositions
\[
\widehat{\mathfrak{n}}=r_{1}\oplus\cdots\oplus r_{m},\quad\mathrm{Langlands-Shahidi}
\]
and 
\[
r_{j}=\bigoplus_{\ell\ge1}V_{\ell}^{\oplus m_{j}(\kappa,\ell)},\quad m_{j}(\kappa,\ell)\in\mathbb{Z}_{\ge0}
\]

\begin{thm}
Let
\[
\mathsf{Q}_{P}(\kappa,s)=\prod_{j=1}^{m}\prod_{\ell\ge1}\left(\frac{js-\frac{\ell+1}{2}}{js+\frac{\ell-1}{2}}\right)^{m_{j}(\kappa,\ell)}
\]
Then $\mathsf{Q}_{P}(\kappa,s)\cdot E_{P}(\pi_{\kappa},s)$ is holomorphic
and nonzero for $\Re(s)\ge0$. 
\end{thm}

\begin{rem*}
In the author's view, it is desirable to represent the poles of Eisenstein
series in formulas like the above. We have given several examples (see \S \ref{sec:new-results}). 
\end{rem*}

\subsection{\label{subsec:Program-general}The general case}

We now leave the Borel case behind and formulate the expectation for
an arbitrary parabolic subgroup.

\subsubsection{The Eisenstein series}

Fix a standard parabolic subgroup $P=N\rtimes M$ of $G$, not necessarily
maximal. Recall that 
\[
M(\mathbb{A})=M(\mathbb{A})^{1}\times A_{M},\quad A_{P}:=A_{M}\simeq(0,\infty)^{\text{co-rank}(M)}
\]
Let $\pi\subset L^{2}(A_{M}M(F)\backslash M(\mathbb{A}))$ be a cuspidal
(irreducible) automorphic representation of $M(\mathbb{A})$. Let
\[
H_{P}:G(\mathbb{A})\to\mathfrak{a}_{P}
\]
be the Harish-Chandra map. For $\Lambda\in\check{\mathfrak{a}}_{P,\mathbb{C}}$,
we have a representation 
\[
\mathbf{1}\otimes\pi\otimes e^{\langle\rho_{P}+\Lambda,H_{P}(\cdot)\rangle}
\]
of $P(\mathbb{A})=N(\mathbb{A})M(\mathbb{A})^{1}A_{M}$. Denote 
\[
\mathrm{Ind}_{P}^{G}(\pi,\Lambda):=\mathrm{Ind}_{P(\mathbb{A})}^{G(\mathbb{A})}\left(\mathbf{1}\otimes\pi\otimes e^{\langle\rho_{P}+\Lambda,H_{P}(\cdot)\rangle}\right)
\]
As usual, one realizes the induced representation $\mathrm{Ind}_{P}^{G}(\pi,\Lambda)$
on a space $\mathcal{A}_{P}(\pi)$ of automorphic forms on $N(\mathbb{A})M(F)\backslash G(\mathbb{A})$
(see \cite{Lapid_remark}, section 2), independent of $\Lambda$ (the
action of $G(\mathbb{A})$ depends on $\Lambda$). 

Cuspidal Eisenstein series 
\[
E_{P}(\pi):\check{\mathfrak{a}}_{P,\mathbb{C}}\to\text{Hom}_{\mathbb{C}}\left(\mathcal{A}_{P}(\pi),\mathcal{A}[G]\right),
\]
\[
E_{P}(\pi,\Lambda):=E_{P}(\pi)(\Lambda)\in\text{Hom}_{G(\mathbb{A})}\left(\mathrm{Ind}_{P}^{G}(\pi,\Lambda),\ \mathcal{A}[G]\right)
\]
is a meromorphic family of intertwining operators. For $\varphi\in\mathcal{A}_{P}(\pi)$,
we write 
\[
E_{P}(\varphi,\Lambda):=E_{P}(\pi)(\Lambda)(\varphi)
\]
to lighten the notation. For meromorphic continuation of these Eisenstein
series, see Bernstein-Lapid \cite{Lapid_Bernstein}. We take this
for granted. 

\subsubsection{The heuristic}

Spectral decomposition rests on deforming integrals of the form
\[
\int_{\sigma+i\check{\mathfrak{a}}_{P}}\widehat{\eta}(\Lambda)E_{P}(\varphi,\Lambda)d\Lambda,\quad\varphi\in\mathcal{A}_{P}(\pi),\ \eta\in C_{c}^{\infty}(A_{P})
\]
with $\sigma$ far in the positive cone in $\check{\mathfrak{a}}_{P}$,
where the Eisenstein series converges. Write $T_{P}$ for the tube
over the closed positive cone in $\check{\mathfrak{a}}_{P}$. 

Exactly as in the Borel case, we expect a factorization
\[
E_{P}(\pi,\Lambda)=D_{P}(\pi,\Lambda)E_{P}^{*}(\pi,\Lambda)
\]
in which $D_{P}(\pi,\Lambda)$ is a \emph{scalar} rational function
--- the rational \emph{Paley--Wiener distribution} --- recording
precisely the poles and zeros of $E_{P}$ that meet $T_{P}$, and
$E_{P}^{*}(\pi,\Lambda)$ is the regularized
Eisenstein series, holomorphic and nonzero as an intertwining operator
throughout. We have written $D_{P}(\pi,\Lambda)$ in two cases:
\begin{itemize}
\item $G$ a split semisimple group, $P=B$ a Borel subgroup, and $\pi=\mathbf{1}$
the trivial representation. 
\item $G=GL(n)$, $P$ a general parabolic subgroup, $\pi$ a general cuspidal
representation. 
\end{itemize}
The central heuristic is that, once $E_{P}(\pi,\Lambda)$ is written
in this form, the contour deformation is governed entirely by the
scalar factor $D_{P}(\pi,\Lambda)$. Concretely:
\begin{itemize}
\item The special points are detected by deforming the integral
\[
\int_{\sigma+i\check{\mathfrak{a}}_{P}}\widehat{\eta}(\Lambda)D_{P}(\pi,\Lambda)d\Lambda,\quad\eta\in C_{c}^{\infty}(A_{P})
\]
as in the examples. 
\item If $\rho_{P,\#}\in T_{P}$ is a point where the discrete spectrum
is supported, then the corresponding discrete spectrum is 
\[
E_{P}^{*}(\pi,\rho_{P,\#})\left(\mathrm{Ind}_{P}^{G}(\pi,\rho_{P,\#})\right)\subset L^{2}\left(G(F)\backslash G(\mathbb{A})\right).
\]
As $[P,\pi]$ ranges over all cuspidal data, these subspaces together
span the discrete spectrum of $L^{2}\left(G(F)\backslash G(\mathbb{A})\right)$.
\end{itemize}
This is the general-case analogue of the Borel picture: the residual
spectrum is, in every case, a regularization of a cuspidal Eisenstein
series at a finite list of distinguished points, and those points
are read off from a single scalar rational function.

\subsubsection{Why the regularization is the right object}

The rational function $D_{P}(\pi,\Lambda)$ above is likely to have
the same shape as the Borel and $GL(n)$ examples above. The parallel
is close enough to expect morphisms of the form
\[
\mathscr{L}_{F}\times SL(2,\mathbb{C})\to\widehat{M},\quad P=N\rtimes M
\]
to play a role in determining and describing where the discrete spectrum
is concentrated. The problem addressed by Langlands, being more fundamental
than that of Arthur's work (a prerequisite, really), might be the
true origin of many of the phenomena observed/conjectured by Arthur. 

Of all the Laurent coefficients we single out the regularization and
discard the rest. I give several reasons for believing this is the
correct object.
\begin{itemize}
\item \emph{Canonicity}. The residual spectrum cannot depend on the path
or the measures chosen to compute residues during the contour deformation.
Only $E_{P}^{*}(\pi,\Lambda)$ admits a description independent of
those choices; every other Laurent coefficient requires a choice of
coordinates.
\item \emph{Intertwining. }The representation-theoretic counterpart of a
Hecke eigenfunction is an intertwining operator. Only the regularization
$E_{P}^{*}(\pi,\Lambda)$ is an intertwining operator; the other Laurent
coefficients are not.
\item \emph{Consistency with the motivation.} The theorems of section \ref{sec:new-results}
on the poles of Eisenstein series can extend beyond the cases treated
there only if the non-cuspidal spectrum arises in exactly this way
--- as a regularization.
\item \emph{Meaningfulness. }A theorem of Franke \cite{Franke_firstPaper}
(Corollary 1 on p. 236) asserts that \emph{every} automorphic form
is a linear combination of Laurent coefficients of Eisenstein series
induced from cusp forms. In light of this, an assertion that \emph{square-integrable}
automorphic forms are \emph{some} combination of Laurent coefficients
satisfying Langlands' square-integrability criterion would be nearly
vacuous. 
\end{itemize}
\begin{rem}
The contour deformation requires certain analytic justifications ---
a special cancellation illustrated for $SL(3)$ in \ref{subsec:Critical-remarks-SL3}
--- and these place \emph{sharp constraints} on where the poles of
$E_{P}(\pi,\Lambda)$ can lie for $\pi$ cuspidal. I have found no
evidence that Langlands recognized this issue in the original notes
\cite{Langlands-Spectral}, for it seems to require his \emph{later
}work \emph{Euler products} \cite{Langlands-EulerProduct}. 
\end{rem}

\bigskip{}

\subsubsection*{Coda}

It is my hope that pursuing this picture will prove fruitful, both
as a route to open problems in the theory of automorphic forms and
as a source of new ones. I end with a quote of Weil that I think is particularly appropriate for the theory of automorphic forms: 
\begin{quotation}
``[With the explosion of literature in mathematics, does it]
mean that mathematics is becoming a science for erudites, and that
it will no longer be possible to do creative work in mathematics until
one has grown gray in the harness, and exhausted from burning the
midnight oil for many years in the company of dusty tomes? \emph{This
would at the same time be a sign of its decline};

...

Therefore, if mathematics is to continue to exist in the way in which
it has manifested itself to its votaries until now, the technical
complications with which more than one of its subjects is now studded,
must be superficial or of only temporary character; \emph{in the future,
as in the past, the great ideas must be simplifying ideas}, the creator
must always be one who clarifies, for himself and for others, the
most complicated tissues of formulas and concepts. ''

--- Andr\'e Weil, The future of Mathematics, 1950. Emphasis added. 
\end{quotation}
\bibliographystyle{plain}
\bibliography{references}

\end{document}